\documentclass{amsart}
\usepackage{array}
\usepackage{tabu}
\usepackage{amssymb}
\usepackage {amsmath}
\usepackage {amsthm}
\usepackage {amscd}
\usepackage {amssymb}
\usepackage {latexsym}
\usepackage{multirow}
\usepackage{verbatim}
\usepackage{mathrsfs}
\usepackage[a4paper,left=3.7cm,right=3.7cm,bottom=3.2cm]{geometry}
\usepackage{color}

\newtheorem{theorem}{Theorem}[section]
\newtheorem{lemma}[theorem]{Lemma}
\newtheorem{prop}[theorem]{Proposition}
\theoremstyle{definition}
\newtheorem{definition}[theorem]{Definition}
\newtheorem{example}[theorem]{Example}

\newtheorem{corollary}[theorem]{Corollary}

\theoremstyle{remark}
\newtheorem{remark}[theorem]{Remark}

\numberwithin{equation}{section}

\newcommand{\Ext}{\operatorname{Ext}}

\newcommand{\Dim}{\operatorname{dim}}

\newcommand{\Pic}{\operatorname{Pic}}
\newcommand{\Sym}{\operatorname{Sym}}
\newcommand{\Ker}{\operatorname{Ker}}

\newcommand{\GCD}{\operatorname{GCD}}
\newcommand{\Coker}{\operatorname{Coker}}
\newcommand{\Hom}{\operatorname{Hom}}
\newcommand{\rango}{\operatorname{rk}}
\newcommand{\p}{{\mathbb{P}^1}}
\newcommand{\pp}{{\mathbb{P}}}
\newcommand{\xe}{{X_e}}
\newcommand{\sO}{\mathcal{O}}
\newcommand{\sE}{\mathcal{E}}
\newcommand{\sF}{\mathcal{F}}

 \newenvironment{sistema}%
  {\left\lbrace\begin{array}{@{}l@{}}}%
  {\end{array}\right.}

\begin{document}

{\obeyspaces\global\let =\ }
\font\tteight=cmtt8
\font\ttten=cmtt10

\newdimen\outputBaseLineSkip
         \outputBaseLineSkip = 10pt 
\newskip\beginOutputSkip
        \beginOutputSkip = 4.5 pt plus .5 pt minus .5 pt
\newskip\endOutputSkip
        \endOutputSkip   = 2.5 pt plus .25 pt minus .25 pt

\def\looserOutput#1{%
  \advance\beginOutputSkip by #1
  \advance\endOutputSkip by #1
}
\def\tighterOutput#1{%
  \advance\beginOutputSkip by -#1
  \advance\endOutputSkip by -#1
}

\def\beginOutput{%
    \par
    \penalty -150
    \penalty -150
    \begingroup
      \def\\{%
          \leavevmode
          \hss
          \endgraf
          \penalty 150
          }
      \ttten
      \parindent = 24pt
      \def\${\char`\$}
      \def\{{\char`\{}
      \def\}{\char`\}}
      \catcode`\_=\the\catcode`a
      \catcode`\^=\the\catcode`a
      \catcode`\#=\the\catcode`a
      \catcode`\~=\the\catcode`a
      \catcode`\&=\the\catcode`a
      \parskip=0pt
      \lineskip=0pt
      \obeyspaces
      }
\def\emptyLine{%
    \penalty -100
    \penalty -100
    }
\def\endOutput{%
    \endgroup
    \par
    \penalty -150
    \penalty -150
    \noindent}

\title{Characterization of Ulrich Bundles on Hirzebruch Surfaces
}

\author{Vincenzo Antonelli}
\address{Dipartimento di Scienze Matematiche, Politecnico di Torino, c.so Duca degli Abruzzi 24, 10129 Torino, Italy}

\email{vincenzo.antonelli@polito.it}

\begin{abstract}
In this work we characterize Ulrich bundles of any rank on polarized rational ruled surfaces over $\p$. We show that every Ulrich bundle admits a resolution in terms of line bundles. Conversely, given an injective map between suitable totally decomposed vector bundles, we show that its cokernel is Ulrich if it satisfies a vanishing in cohomology. As a consequence we obtain, once we fix a polarization, the existence of Ulrich bundles for any admissible rank and first Chern class. Moreover we show the existence of stable Ulrich bundles for certain pairs $(\rango(E),c_1(E))$ and with respect to a family of polarizations. Finally we construct examples of indecomposable Ulrich bundles for several different polarizations and ranks.
\end{abstract}
\maketitle

\section{Introduction}
Through this paper we will work over $k$, an algebraically closed field of characteristic $0$, and we will denote $\pp^N$ the projective space over $k$ of dimension $N$. 

Let $X\subset \pp^N$ be a projective variety which is naturally endowed with the very ample line bundle $\sO_X(h)=\sO_X \otimes \sO_{\pp^N}(1)$. We say that $X$ is arithmetically Cohen-Macaulay (aCM for short) if $H^i(X,\mathcal{I}_X(th))=0$ for $t\in\mathbb{Z}$ and $1\leq i \leq \Dim(X)$. 

A coherent sheaf $E$ over a projective variety $X$ is called aCM if it is locally CM and all of its intermediate cohomology groups vanish, i.e. $H^i(X,E(th))=0$ for $1\leq i\leq \Dim(X)-1$. The study of aCM sheaves supported on $X$ gives us a measure of the complexity of the variety itself.  In this paper we will focus on a particular family of aCM sheaves, namely the Ulrich ones. They are defined to be the aCM sheaves whose corresponding module of twisted global sections has the maximum number of generators.

In \cite{eisenbud1} D. Eisenbud and F.O. Schreyer characterized Ulrich sheaves $E$ with respect to a very ample line bundle $\sO_X(h)$ as the ones such that
\[
h^i(X,E(-ih))=h^{j}(X,E(-(j+1)h))=0
\]
for each $i>0$ and $j<\Dim(X)$. Ulrich sheaves carry many interesting properties and they are the simplest bundles from the cohomological point of view. Furthermore an Ulrich sheaf is locally free on the complement of the singular locus of the variety $X$. In \cite{eisenbud1} the authors raised the following questions: 

Is every variety $X \subset \pp^N$ the support of an Ulrich sheaf? If so, what is the smallest possible rank for such a sheaf?

The answers to these questions are still unknown in general, although there are several scattered results: e.g. \cite{beauville} for an introduction on Ulrich bundles and some results on surfaces and threefolds, \cite{aprodufarkas} and \cite{casnatigalluzzi} for $K3$ surfaces, \cite{malaspina} for smooth projective varieties of minimal degree, \cite{cashart} for non-singular cubic surfaces, \cite{casnati4} for non-special surfaces with $p_g=q=0$, \cite{coskun} and \cite{casnatianti} for del Pezzo surfaces, \cite{COSKUN201749} for flag varieties and \cite{Costa2015} for Grassmannians.

The aim of this paper is to give an intrinsic characterization of Ulrich bundles over Hirzebruch surfaces. In \cite{Aprodu} M. Aprodu, L. Costa and R. M. Mir\'o-Roig discussed the existence of Ulrich line bundles and special rank two Ulrich bundles over ruled surfaces, which are the rank two Ulrich bundles $E$ such that $\det(E)=3h+K_X$. Their existence implies that the associated Cayley-Chow form is represented as a linear pfaffian \cite{eisenbud1}. The authors proved that Ulrich line bundles over ruled surfaces exist only for a particular choice of the polarization $\sO_X(h)$, and they proved the existence of special Ulrich bundles under a mild assumption on the polarization. 

 In \cite{malaspina2} and \cite{miro3} the authors considered the case of Hirzebruch surfaces embedded as rational normal scrolls. 
 
Recall that a rational ruled surface is a surface $X$ together with a surjective morphism $\pi:X\rightarrow \p$, such that the fibre $X_y$ is isomorphic to $\p$ for every point $y\in \p$ and such that $\pi$ admits a section. For each $e\geq 0$ there is exactly one rational ruled surface with invariant $e$ over $\p$, given by $\pi: \xe=\mathbb{P}(\sO_\p \oplus \sO_\p(-e)) \rightarrow \p$. The Picard group is given by $\Pic(\xe)\cong\mathbb{Z}\oplus \pi^\ast \Pic (\p)$, generated by a section $ C_0$ and a fiber $f$ with the intersection relations $ C_0^2=-e$, $ C_0 f =1$ and $f^2=0$. Given a divisor $D=a C_0+bf$ we will write the associated line bundle as $\sO_\xe(a,b)$ or $\sO_\xe(a C_0 + bf)$. A divisor $D=a C_0+bf$ is very ample if and only if $a>0$ and $b>ae$. Furthermore the canonical divisor on $\xe$ is given by $K_\xe=-2 C_0 -(2+e)f$. In what follows we will denote the polarized Hirzebruch surface by $(\xe,\sO_\xe(a,b))$.

In this paper we prove the following theorem:
\begin{theorem}\label{teoremaprincipale}
Let $(\xe,\sO_{\xe}(a,b))$ be a polarized Hirzebruch surface and $E$ a rank $r$ Ulrich bundle with $c_1(E)=\alpha  C_0 + \beta f$. 
\begin{enumerate}
\item Then $E$ fits into a short exact sequence of the form
\begin{equation*}
0\rightarrow \sO_\xe^\gamma (a-1,b-e-1){\rightarrow}\sO_\xe^\delta (a-1,b-e)\oplus \sO_\xe^\tau(a,b-1) \rightarrow E \rightarrow 0
\end{equation*}
where $\gamma=\alpha+\beta-r(a+b-1)-e(\alpha-ar)$, $\delta=\beta-r(b-1)-e(\alpha-ar)$ and $\tau=\alpha-r(a-1)$.
\item Then $E$ fits into a short exact sequence of the form
\begin{equation*}
0 \rightarrow E\rightarrow  \sO_\xe^\lambda(2a-1,2b-2) \oplus \sO_\xe^\mu (2a-2,2b-1-e){\longrightarrow} \sO_\xe^\nu (2a-1,2b-1)\rightarrow 0
\end{equation*}
where $\lambda=r(2b-1-e)-\beta-e(r(2a-2)-\alpha)$, $\mu=r(2a-1)-\alpha$ and $\nu= r(2a+2b-3-e)-\alpha-\beta -e(r(2a-2)-\alpha)$. 
\end{enumerate}
\end{theorem}
Thus we are able to express $E$ as the cokernel (resp. kernel) of a certain injective (resp. surjective) map, using Beilinson's spectral sequence techniques. A result of this type was obtained in \cite{coskun2} for the Veronese surface, and in \cite{lin} for the projective space $\pp^N$ embedded with a very ample line bundle $\sO_{\pp^N}(d)$. 

In \cite{hirzebruch} I. Coskun and J. Huizenga found a similar resolution with totally different methods. They used it to classify Chern characters such that the correspondent general stable bundle on $\xe$ is globally generated.

Then we discuss the inverse problem: given an injective map $\phi$ as above, is $E=\Coker(\phi)$ an Ulrich bundle? Our answer is given by

\begin{theorem}\label{condizioneinversa}
Let $(\xe,\sO_\xe(h))$ with $h=a C_0+bf$ be a polarized Hirzebruch surface. 
\begin{enumerate}
\item Let $\phi$ be an injective map 
\begin{equation*}
 \sO_\xe^\gamma (a-1,b-e-1) \overset{\phi}{\longrightarrow} \sO_\xe^\delta (a-1,b-e)\oplus \sO_\xe^\tau(a,b-1)
\end{equation*}
with $\delta,\tau$ non-negative, $\gamma$ positive and $\delta+\tau>\gamma$. Let us call $r=\delta+\tau-\gamma$ and denote with $E$ the cokernel of $\phi$. In particular $c_1(E)=\alpha C_0+\beta f$ with $\alpha=\tau+r(a-1)$ and $\beta=r(b-1)+\delta-e(r-\tau)$. If $c_1(E)h=\frac{r}{2} (3h^2+hK_\xe)$ then $E$ is an Ulrich bundle if and only if $H^2(\xe,E(-2h))=0$.
\item If $\psi$ is a surjective map
\begin{equation*}
\sO_\xe^\lambda(2a-1,2b-2) \oplus \sO_\xe^\mu (2a-2,2b-1-e)\overset{\psi}{\longrightarrow} \sO_\xe^\nu (2a-1,2b-1)
\end{equation*}
with $\lambda,\mu$ non-negative, $\nu$ positive and $\lambda+\mu>\nu$. Let us call $r=\lambda+\mu-\nu$ and denote by $E$ the kernel of $\psi$. In particular $c_1(E)=\alpha C_0+\beta f$ with $\alpha=r(2a-1)-\mu$ and $\beta=r(2b-1-e)-\lambda-e(\mu-r)$. If $c_1(E)h=\frac{r}{2} (3h^2+hK_\xe)$ then $E$ is an Ulrich bundle if and only if $H^0(\xe,E(-h))=0$.
\end{enumerate}\end{theorem}

For the first case, by computing the long exact sequence in cohomology we see that $h^0(\xe,E(-h))=h^1(\xe,E(-h))=0$ and that $h^1(\xe,E(-2h))=h^2(\xe,E(-2h))$. Thus as soon as one of them vanishes, also the other does, and this is equivalent to the injectivity of the induced map $H^2(\phi(-2h))$. The second case is treated in an analogous way.

In the particular case of special Ulrich bundles, the induced map $H^2(\phi(-2h))$ is always injective. So we obtain an alternative (with respect to \cite[Theorem 3.4]{Aprodu}) proof of the existence of special Ulrich bundles, characterizing them as the cokernel of a map between very well understood vector bundles. We offer a family of counterexamples which shows that the cohomological condition $h^1(\xe,E(-2h))=0$ is necessary in general, and we describe the locus where $\Coker (\phi)$ fails to be Ulrich. Furthermore, using Theorem \ref{teoremaprincipale} and \ref{condizioneinversa}, we are able to find numerical conditions that the pair $(r,D)$ must satisfy in order to be admissible for a rank $r$ Ulrich bundle $E$ with $c_1(E)=D$ (see Definition \ref{definitionadmissible}). We will call such pairs {\emph{admissible Ulrich pairs}.

Using Theorem \ref{condizioneinversa} we are able to prove the following existence theorem:

\begin{theorem}\label{esistenza}
Let us consider $(\xe,\sO_\xe(h))$ with $h=aC_0+bf$. If the map $\phi$ as in Theorem \ref{condizioneinversa} is general, then $E=\Coker (\phi)$ is Ulrich. In particular on $(\xe,\sO_\xe(a,b))$ there exist Ulrich bundles for any admissible Ulrich pair.
\end{theorem}
The strategy is to use \cite[Theorem 1.1]{hirzebruch} to obtain that the bundle $E$, realized as the cokernel of a general map $\phi$, has the property that $E(-2h)$ has natural cohomology. Since $h^1(\xe,E(-2h))=h^2(\xe,E(-2h))$ they cannot be both different from zero and thus $E$ is Ulrich. This shows that every Hirzebruch surface admits Ulrich bundles of any admissible first Chern class, any rank and with respect to every very ample polarization.

Then we focus on stable Ulrich bundles on Hirzebruch surfaces $(\xe,\sO_\xe(a,b))$ with $\GCD(a,b)=1$, $e>0$ and $a>1$. We obtain the following existence result:
\begin{prop}
Let us consider $(\xe,\sO(a,b))$ with $\GCD(a,b)=1$, $e>0$ and $a>1$. Suppose $a$ is odd or $e$ is even. Then for any even rank $r<\frac{b}{b-ea+e}$ there exists a stable rank $r$ Ulrich bundle $E$ with $c_1(E)=\frac{r}{2}(3h+K_\xe)$.
\end{prop}

Furthermore, using the computer software \emph{Macaulay2}, Theorem \ref{condizioneinversa} gives us a useful tool to construct examples of Ulrich bundles of higher rank with a fixed first Chern class. 

In what follows we describe the structure of the paper. In Section 2 we summarize some general results on Ulrich bundles on polarized surfaces, on rational ruled surfaces and we recall some facts about full exceptional collections of sheaves. In Section 3 we prove Theorem \ref{teoremaprincipale} and we give an alternative presentation of an Ulrich bundle as the cohomology of a monad. In Section 4 we prove Theorem \ref{condizioneinversa}. In Section 5 we study the admissible ranks and first Chern classes of an Ulrich bundle. In Section 6 we prove Theorem \ref{esistenza}, we deal with the existence of stable Ulrich bundles and we study their moduli spaces. Finally, in Section 7 we construct several different examples of indecomposable Ulrich bundles with respect to a generic very ample polarization.


\section{Preliminaries and general results}

We recall some facts about ruled surfaces $\xe$ on $\p$ with invariant $e\geq 0$ (for details on ruled surfaces see \cite[Section V.2]{hart}). 

On $\xe$ there are two natural short exact sequences. The first one is 
\begin{equation}\label{eulerop1}
0 \rightarrow \sO_\xe(0,-1)\rightarrow \sO_\xe^2\rightarrow \sO_\xe(0,1)\rightarrow 0,
\end{equation}
which is the pullback on $\xe$ of the Euler sequence over $\p$. Other than that we also have a second natural exact sequence
\begin{equation*}
0\rightarrow \Omega_{\xe/\p} \rightarrow \pi^\ast (\sO_{\p}\oplus \sO_{\p}(-e))\otimes \sO_\xe(-1,0)\rightarrow \sO_\xe\rightarrow 0
\end{equation*}
which,  in this case, will take the form
\begin{equation}\label{eulero2}
0\rightarrow \sO_\xe(-1,-e)\rightarrow \sO_\xe\oplus \sO_\xe(0,-e)\rightarrow \sO_\xe(1,0)\rightarrow 0.
\end{equation}

Now we recall how to compute the cohomology of the line bundles over $\xe$.

\begin{lemma}\cite[Exercise III.8.1, III.8.4]{hart}\label{cohomologyrul}
Given $\sO_\xe(t C_0+sf)$ a line bundle on $\pi: \xe\rightarrow \p$ then
\[
H^i(\xe,\sO_\xe(t C_0+sf))=
\begin{sistema}
0 \ \text{if $t=-1$} \\
H^i(\p, \Sym^t(\mathcal{E})\otimes \sO_{\p}(s)) \ \text{if $t\geq0$} \\
H^{2-i}(\p,\Sym^{-2-t}(\mathcal{E})\otimes \sO_{\p}(-e-s-2)) \ \text{if $t\leq -2$},
\end{sistema}
\]

where $\mathcal{E}=\sO_{\p} \oplus \sO_{\p}(-e)$.

\end{lemma}
Now we briefly recall some basic definitions about derived categories of sheaves that we need in order to state the Beilinson's theorem (for the notation and more details see \cite{helix}).
\begin{definition}
Given a smooth projective variety $X$, let $D^{b}(X)$ be the bounded derived category of sheaves on $X$.
\begin{itemize}
\item[\ a)] An object $E \in D^b(X)$ is called {\em{exceptional}} if $\Ext^\bullet (E,E)=k$.
\item[\ b)] A set of exceptional objects $(E_0,\dots,E_n)$ is called an {\em{exceptional collection}} if $\Ext^\bullet (E_i,E_j)=0$ for all $i>j$.
\item[\ c)] An exceptional collection is said to be {\em{full}} when $\Ext^\bullet(E_i,A)=0$ for all $i$ implies $A=0$.
\end{itemize}
\end{definition}

In the next sections we will use the Beilinson's spectral sequence in order to obtain a resolution of Ulrich bundles over $\xe$ (for this version of the theorem see \cite[Theorem 2.5]{malaspina}).
\begin{theorem}[Beilinson-type spectral sequence]\label{Beilinson}
Let $X$ be a smooth projective variety with a full exceptional collection $(E_0,\dots,E_n)$ where $E_i=\mathcal{E}_i^{\vee}\left[ -k_i \right]$, with each $\mathcal{E}_i$ a vector bundle, and $(k_0,\dots,k_n) \in \mathbb{Z}^{n+1}$ such that there exists a sequence $(\mathcal{F}_0=F_0,\dots,\mathcal{F}_n=F_n)$ of vector bundles satisfying
\begin{equation}\label{dual}
\Ext^k(E_i,F_j)=H^{k+k_i}(X,\mathcal{E}_i \otimes \mathcal{F}_j)=
\begin{cases}
\mathbb{C} \ & \text{if $i=j=k$,} \\
0 \ & \text{otherwise.}
\end{cases}
\end{equation}
$(F_n,\dots,F_0)$ is the right dual collection of $(E_0,\dots,E_n)$. Then for every coherent sheaf $A$ on $X$ there is a spectral sequence in the square $-n \le p \le 0, \ 0\le q \le n$ with $E_1$-term
\[
E_1^{p,q}=H^{q+k_{-p}}(X,\mathcal{E}_{-p}\otimes A)\otimes \mathcal{F}_{-p}
\]
which is functorial in $A$ and converges to
\[
E_{\infty}^{p+q}=
\begin{cases}
A \ & \text{if $p+q=0$,} \\
0 \ & \text{otherwise.}
\end{cases}
\]
\end{theorem}
\begin{remark}\label{rembeil}
It is possible to state a stronger version of the Beilinson's theorem (for details see \cite{ottanc}, where the theorem is stated for $\mathbb{P}^N$). Let us consider an Hirzebruch surface $X_e$ and let $A$ be a coherent sheaf on $X_e$. Let $(E_0,\dots,E_n)$ be a full exceptional collection and $(F_n,\dots,F_0)$ its right dual collection. If $\Ext^k(F_i,F_j)=0$ for $k>0$ and all $i,j$  then there exists a complex of vector bundles $L^\bullet$ such that
\begin{enumerate}
\item $H^k(L^\bullet)= 
\begin{cases}
A \ & \text{if $k=0$},\\
0 \ & \text{otherwise}.
\end{cases}$ 
\item $L^k=\underset{k=p+q}{\bigoplus}H^{q+k_{-p}}(\xe,A\otimes \sE_{-p})\otimes \sF_{-p}$ with $0\le q \le n$ and $-n\le p \le 0$. 
\end{enumerate} 
\end{remark}

In the next sections we will apply Beilinson's theorem to sheaves on $\xe$, so we need to compute an exceptional collection for these surfaces.

\begin{prop}
The collection
\begin{equation}\label{coll1}
(\sO_{\xe}(-1,-1),\sO_{\xe}(-1,0),\sO_{\xe}(0,-1),\sO_{\xe})
\end{equation}
is a full exceptional collection for $\xe$, whose dual is
\begin{equation}\label{coll2}
(\sO_{\xe}(-1,-e-1),\sO_{\xe}(-1,-e),\sO_{\xe}(0,-1),\sO_{\xe}).
\end{equation}
\end{prop}
\begin{proof}
For \eqref{coll1} see \cite[Corollary 2.7]{orlov}. For \eqref{coll2} use the duality condition expressed in \eqref{dual}. 
\end{proof}

Now we recall the definition of Ulrich bundles.

Let $X\subset \mathbb{P}^N$ be a smooth irreducible closed variety, let $F$ be a vector bundle on $X$ and let $\sO_X(h)$ be the induced polarization. We say that:
\begin{itemize}
\item [$\bullet$]$F$ is {\em{initialized}} if $h^0(X,F(-h))=0\neq h^0(X,F)$.
\item [$\bullet$]$F$ is {\em{aCM}} if $h^i(X,F(th))=0$ for each $t\in \mathbb{Z}$ and $i=1,\dots, \Dim(X)-1$.
\item [$\bullet$]$F$ is {\em{Ulrich}} if $h^i(X,F(-ih))=h^j(X,F(-(j+1)h))=0$ for each $i>0$ and $j<\Dim (X)$.
\end{itemize}
Ulrich bundles carry many properties. E.g. they are initialized, aCM and globally generated. Every direct summand of an Ulrich bundle is Ulrich as well. They also well behave with respect to the notions of (semi)stability and $\mu$-(semi)stability. Recall that for each bundle $F$ on $X$, the slope $\mu(F)$ and the reduced Hilbert polynomial $p_F(t)$ (with respect to $\sO_X(h)$) are defined as follows:
\[
\mu(F)=c_1(F)h^{\Dim(X)-1}/\rango(F), \qquad p_F(t)=\chi(F(th))/\rango(F).
\]
The bundle $F$ is $\mu$-semistable (resp. $\mu-stable$) if for all proper subsheaves $G$ with $0<\rango(G)<\rango(F)$ we have $\mu(G) \le \mu(F)$ (resp. $\mu(G)<\mu(F)$).

The bundle $F$ is called semistable (resp. stable) if for all $G$ as above $p_G(t)\leq p_F(t)$ (resp. $p_G(t)<p_F(t)$) for $t>>0$. We continue with some general results on Ulrich bundles.

\begin{theorem}\cite[Theorem 2.9]{cashart}\label{ulrichch}
Let $X\subset \mathbb{P}^N$ be a smooth, irreducible, closed variety. If $E$ is an Ulrich bundle on $X$ then the following assertions hold.
\begin{itemize}
\item[ a)] $E$ is semistable and $\mu$-semistable.
\item[ b)] $E$ is stable if and only if it is $\mu$-stable.
\item[ c)] if 
\[ 
0 \rightarrow L \rightarrow E \rightarrow M \rightarrow 0
\] 
is an exact sequence of coherent sheaves with $M$ torsion free and $\mu(L)=\mu(E)$, then both $L$ and $M$ are Ulrich bundles.
\end{itemize}
\end{theorem}

\begin{prop}\cite[Proposition 2.1]{casnati4}\label{lemmachernruled}
Let $S$ be a surface endowed with a very ample line bundle $\sO_S(h)$. If $E$ a vector bundle on $S$, then the following are equivalent:
\begin{enumerate}
\item [$\bullet$]$E$ is an Ulrich bundle;
\item [$\bullet$]$E^\vee(3h+K_S)$ is an Ulrich bundle;
\item [$\bullet$] $E$ is an aCM bundle and
\begin{equation}\label{eq}
\begin{gathered}
c_1(E)h=\frac{\rango (E)}{2}(3h^2+hK_S), \\
c_2(E)=\frac{1}{2}(c_1(E)^2-c_1(E)K_S)- \rango(E)(h^2-\chi (\sO_S));
\end{gathered}
\end{equation}
\item [$\bullet$] $h^0(S, E(-h))=h^0(S,E^\vee(2h+K_S))=0$ and equalities \eqref{eq} hold.
\end{enumerate}
\end{prop}
Moreover, the Riemann-Roch theorem on a surface $S$ is 
\begin{equation}\label{RR}
\chi(F)=\frac{c_1^2(F)}{2}-\frac{c_1(F)K_S}{2}-c_2(F)+\rango(F)\chi (\sO_S)
\end{equation}
for each locally free sheaf $F$ on $S$.

We conclude this section with some results about Ulrich bundles on ruled surfaces. The following results have been obtained in a more general context (see \cite{Aprodu}). However, we will state them in the case of Hirzebruch surfaces. We start with existence of Ulrich line bundles.
\begin{prop}\cite[Theorem 2.1]{Aprodu}\label{linebundlesruled}
Let us consider $(\xe,\sO_\xe(h))$ with $h=aC_0+bf$ and $e>0$. Then there exists Ulrich line bundles with respect to $h$ if and only if $a=1$, and they are
\[
\sO_{\xe}(0,2b-1-e) \qquad \text{and} \qquad \sO_{\xe}(1,b-1).
\]
\end{prop}
Now we consider rank two Ulrich bundles. 
\begin{prop}\cite[Proposition 3.1, 3.3, Theorem 3.4]{Aprodu}\label{r2ruled1}
Let us consider $(\xe,\sO_\xe(h))$ with $h=aC_0+bf$ and $e>0$. Then
\begin{enumerate}
\item if $a=1$ there exists a family of dimension $2b-e-3$ of indecomposable, rank-two, simple, strictly semistable, special Ulrich bundles on $\xe$.
\item if $a \geq 2$ there exists special Ulrich bundles with respect to $h$ given by extensions
\[
0\rightarrow \sO_{\xe} (a,b-1)\rightarrow E \rightarrow \mathcal{I}_Z(2a-2,2b-1-e)\rightarrow 0,
\] 
where $Z$ is a general zero-dimensional subscheme of $\xe$ with $l(Z)=(a-1)(b-\frac
{ea}{2})$.
\end{enumerate} 
\end{prop}


\section{From the Ulrich bundle to the resolution}
We start this section by describing the cohomology of an Ulrich bundle $E$ on $\xe$.
\begin{lemma}\label{cohomologyulrich}
Let $E$ be a rank $r$ Ulrich bundle on $(\xe,\sO_\xe(a,b))$. Then
\begin{enumerate}
\item $h^0(\xe,E(t,s))=h^2(\xe, E(t,s))=0$ for all $-2a\leq t \leq -a$ and $-2b\leq s \leq -b$
\item $h^1(\xe, E(-a,s))=h^2(\xe, E(-a,s))=0$ for all $s\geq -b$.
\item $h^1(\xe, E(t,s))=h^2(\xe,E(t,s))=0$ for all $t\geq-a$ and $s\geq -b+e$.
\item $h^0(\xe, E(t,s))=h^1(\xe, E(t,s))=0$ for all $t \leq -2a$ and $s \leq -2b-e$.
\end{enumerate}
In particular $E(t,s)$ has natural cohomology (i.e. there exists at most one $k$ such that $h^k(E(t,s))\neq 0$) for such $t$ and $s$.
\end{lemma}
\begin{proof}
 For the first part of the lemma, since $E$ is Ulrich, then $E(-a,-b)$ has no cohomology. For any effective divisor $D$ we have the following short exact sequence
 \begin{equation}\label{diveffettivo}
 0\rightarrow \sO_\xe(-D)\rightarrow \sO_{\xe}\rightarrow \sO_{D}\rightarrow 0.
 \end{equation}
Tensoring \eqref{diveffettivo} by $E(-a,-b)$ and considering the long exact sequence in cohomology, we obtain $h^0(\xe,E(t,s))=0$ for all $t \leq -a$ and $ s \leq -b$. Using Serre's duality and the fact that $E^\vee(3h+K_{\xe})$ is Ulrich, we obtain $h^2(\xe,E(t,s))=0$ for all $t \geq -2a$ and $ s \geq -2b$.

For the second part, we proceed by induction on $s$. We have $h^1(\xe,E(-a,-b))=0$ because $E$ is Ulrich. Suppose $h^1(\xe,E(-a,k))=0$ for all $-b\leq k \leq s $, tensor \eqref{eulerop1} by $E(-a,s)$ and consider the long exact sequence in cohomology. Since $h^1(\xe,E(-a,s))=0$ by inductive hypothesis, we have $h^1(\xe,E(-a,s+1))=0$, which proves (2).

For the third part we want to show that $h^1(\xe, E(t,s))=0$ for all $t\geq-a$ and $s\geq -b+e$, so we proceed by a double induction on $t$ and $s$. Suppose $s\geq-b$, by (2) we have that $h^1(\xe,E(-a,s))=0$. Suppose that $h^1(\xe, E(k,s))=0$ for all $-a\leq k \leq t$ and $s\geq- b$, and tensor \eqref{eulero2} by $E(t,s)$. Considering the long exact sequence induced in cohomology we have that if $s\geq -b+e$, then $h^1(\xe,E(t,s))=h^1(\xe,E(t,s-e))=0$  by inductive hypothesis, and $h^2(\xe,E(t-1,s-e))=0$ because $E$ is Ulrich, so we can conclude that $h^1(\xe,E(t+1,s))=0$ which proves (3).

For the last part recall that since $E$ is Ulrich then the same holds for $E^\vee (3h+K_{\xe})$, so we obtain (4) using (3) and Serre's duality.

\end{proof}

Now we will prove one of the main theorems of this work. 

\begin{theorem}\label{mainth}
Let $E$ be an Ulrich bundle of rank $r$ on $(\xe,\sO_\xe(h))$ with $h=a C_0+bf$ and with first Chern class $c_1(E)=\alpha  C_0+\beta f$.
\begin{enumerate}
\item Then $E$ fits into a short exact sequence of the form
\begin{equation}\label{risoluzionerul}
0\rightarrow \sO_\xe^\gamma (a-1,b-e-1){\rightarrow}\sO_\xe^\delta (a-1,b-e)\oplus \sO_\xe^\tau(a,b-1) \rightarrow E \rightarrow 0
\end{equation}
where $\gamma=\alpha+\beta-r(a+b-1)-e(\alpha-ar)$, $\delta=\beta-r(b-1)-e(\alpha-ar)$ and $\tau=\alpha-r(a-1)$.
\item Then $E$ fits into a short a exact sequence of the form
\begin{equation}\label{risoluzionerulkernel}
0 \rightarrow E\rightarrow  \sO_\xe^\lambda(2a-1,2b-2) \oplus \sO_\xe^\mu (2a-2,2b-1-e)\overset{\psi}{\longrightarrow} \sO_\xe^\nu (2a-1,2b-1)\rightarrow 0
\end{equation}
where $\lambda=r(2b-1-e)-\beta-e(r(2a-2)-\alpha)$, $\mu=r(2a-1)-\alpha$ and $\nu= r(2a+2b-3-e)-\alpha-\beta -e(r(2a-2)-\alpha)$. 
\end{enumerate}
\end{theorem}

\begin{proof}
We apply the Beilinson's Theorem to $E(-a,-b)$ as in Remark \ref{rembeil}. We start by computing the cohomology table of $E(-a,-b)$.
\begin{center}
\renewcommand\arraystretch{2} 
\begin{tabular}{ccccl}
\multicolumn{1}{c}{$\sO_\xe(-1,-e-1)$}                    & \multicolumn{1}{c}{$\sO_\xe(-1,-e)$}        & \multicolumn{1}{c}{$\sO_\xe(0,-1)$}        & \multicolumn{1}{c}{$\sO_\xe$}            &    \\ \cline{1-4}
\multicolumn{1}{|c|}{0} & \multicolumn{1}{c|}{0}        & \multicolumn{1}{c|}{0}        & \multicolumn{1}{c|}{0}            & $h^3$ \\ \cline{1-4}
\multicolumn{1}{|c|}{$\gamma$} & \multicolumn{1}{c|}{$\delta$}        & \multicolumn{1}{c|}{0}        & \multicolumn{1}{c|}{0}            & $h^2$ \\ \cline{1-4}
\multicolumn{1}{|c|}{0} & \multicolumn{1}{c|}{0} & \multicolumn{1}{c|}{$\tau$}        & \multicolumn{1}{c|}{0}            & $h^1$ \\ \cline{1-4}
\multicolumn{1}{|c|}{0} & \multicolumn{1}{c|}{0}        & \multicolumn{1}{c|}{0} & \multicolumn{1}{c|}{0} & $h^0$ \\ \cline{1-4}
\multicolumn{1}{c}{ $E(-a-1,-b-1)[-1]$}                & \multicolumn{1}{c}{$E(-a-1,-b)[-1]$}  & \multicolumn{1}{c}{$E(-a,-b-1)$}  & \multicolumn{1}{c}{$E(-a,-b)$}      &   
\end{tabular}
\label{tabella1}
\end{center}
Every column represents the dimension of the cohomology groups of the vector bundle at the bottom. The therms on the top of the table are the vector bundles that will appear in the Beilinson resolution as in Remark \ref{rembeil}. The shifts in the last two columns represent the $k_i$'s in Theorem \ref{Beilinson} and Remark \ref{rembeil}.

Since $\xe$ is a surface, it follows that:
\[
h^3(\xe,E(-a,-b))=h^3(\xe,E(-a,-b-1))=0
\]
and trivially
\[
h^0(\xe,E(-a-1,-b-1))[-1]=h^0(\xe,E(-a-1,-b))[-1]=0.
\]
All the zeroes in the table are obtained using Lemma \ref{cohomologyulrich}. Since all the vector bundles in the table have natural cohomology, we will use the Riemann-Roch theorem to compute the only non-zero cohomology groups. In general given a divisor $D$ on $\xe$ we have
\begin{gather*}
c_1(E(D))=c_1(E)+rD\\
c_2(E(D))=c_2(E)+(r-1)c_1(E)D+\frac{r(r-1)}{2}D^2.
\end{gather*}
So by Riemann-Roch and using Proposition \ref{lemmachernruled} we have
\begin{equation}\label{RRuled}
\chi(E(D))=rh^2+c_1(E)D+\frac{r}{2}D(D-K_{\xe}).
\end{equation}
\begin{itemize}

\item [$\bullet$]{\textbf {$E(-a-1,-b-1)$}}.

In this case $D=-h- C_0-f$ so using Proposition \ref{lemmachernruled} and equality \eqref{RRuled} we have
\begin{equation*}
\chi(E(D))=-\alpha-\beta+r(a+b-1)+e(\alpha-ra).
\end{equation*}

\item[$\bullet$]{\textbf {$E(-a-1,-b)$}}.

In this case $D=-h- C_0$ so as above
\begin{equation*}
\chi(E(D))=-\beta+r(b-1)+e(\alpha-ar).
\end{equation*}
\item [$\bullet$]{\textbf {$E(-a,-b-1)$}}.

In this case $D=-h+f$ and
\begin{equation*}
\chi(E(D))=-\alpha +r(a-1).
\end{equation*}
\end{itemize}
By Theorem \ref{Beilinson} and Remark \ref{rembeil} we have a short exact sequence
\begin{equation}
0\rightarrow \sO_\xe^\gamma (-1,-1-e) \rightarrow \sO_\xe^\delta (-1,-e)\oplus \sO_\xe^\tau(0,-1) \rightarrow E(-a,-b) \rightarrow 0
\end{equation}
with $\gamma=\alpha+\beta-r(a+b-1)-e(\alpha-ar)$, $\delta=\beta-r(b-1)-e(\alpha-ar)$ and $\tau=\alpha-r(a-1)$. Tensoring the sequence by $\sO_\xe(h)$ we obtain part (1) of the theorem.

For part (2) recall that if $E$ is Ulrich, then the same is true for $E^\vee(3h+K_{\xe})$. Applying Beilinson's theorem to $E^\vee(3h+K_{\xe})$, and dualizing the sequence, we obtain (2).
\end{proof}
Observe that Theorem \ref{mainth} imposes some numerical necessary conditions that a vector bundle must satisfy in order to be Ulrich.
\begin{corollary}\label{condizioninumeriche}
Let $E$ be a rank $r$ vector bundle on $(\xe,\sO_\xe(a,b))$ with first Chern class $c_1(E)=\alpha  C_0 + \beta f $. If $E$ is Ulrich then, using the notation of Theorem \ref{mainth}, $\delta$, $\tau$ and $\gamma$ (resp. $\lambda$, $\mu$ and  $\nu$) are non-negative, with $\delta$ and $\tau$ (resp. $\lambda$ and $\mu$) not both zero. Moreover if $a>1$ and $e>0$ then $\gamma$ and $\tau$ (resp. $\mu$ and $\nu$) are positive.
\end{corollary}
\begin{proof}
The non-negativity follows directly from the fact that the exponents of the resolutions \eqref{risoluzionerul} and \eqref{risoluzionerulkernel} must be non-negative since they represent the dimension of a cohomology group. Since $\delta+\tau-\gamma=r>0$, $\delta$ and $\tau$ cannot be both zero. Furthermore, suppose $\gamma=0$. Then we will have $E\cong \sO_\xe^\delta(a-1,b-e)\oplus \sO_\xe^\tau(a,b-1)$ but by Propositions \ref{linebundlesruled} and \ref{linebundlesX0}, if $e=0$ then both $\sO_\xe^\delta(a-1,b-e)$ and $\sO_\xe^\tau(a,b-1)$ are Ulrich. If $e>1$ then $\sO_\xe^\delta(a-1,b-e)$ is not Ulrich and $\sO_\xe^\tau(a,b-1)$ is Ulrich only when $a=1$. So $\gamma$ can be zero only if $e=0$ or $a=1$ (and in this case also $\delta=0$). The part regarding $\nu$ is completely analogous. Now we prove the last part of the statement. Suppose $e>0$ and $a>1$. If $\tau=0$ then $E$ would be the pull-back of a vector bundle of $\mathbb{P}^1$. By Grothendieck's theorem every vector bundle on $\mathbb{P}^1$ is the direct sum of line bundles, i.e $E=\pi^\ast(\bigoplus_{i}{L_i})$. Since $E$ is Ulrich, each $\pi^\ast(L_i)$ is Ulrich, but this is not possible since, by Proposition \ref{linebundlesruled}, $\xe$ does not admit Ulrich line bundles if $e>0$ and $a>1$.
\end{proof}
Using similar techniques it is also possible to retrieve each Ulrich bundle on $\xe$ as the cohomology of a monad.
\begin{prop}
Let $E$ be a rank $r$ Ulrich bundle on $(\xe,\sO_\xe(a,b))$ with $a>1$ and with first Chern class $c_1(E)=\alpha  C_0+\beta f$. Then $E$ is the cohomology of a monad of the form
\begin{equation}\label{monad}
0\rightarrow \sO_{\xe}^\varepsilon(a-1,b-e)\rightarrow \sO_{\xe}^\zeta(a-1,b+1-e) \oplus \sO_{\xe}^\eta(a,b)\rightarrow \sO_{\xe}^\vartheta(a,b+1)\rightarrow 0
\end{equation}
where $\varepsilon=2\alpha+\beta-r(2a+b-1)-e(\alpha-ar)$, $\zeta=2\alpha-2r(a-1)$, $\eta=\alpha+\beta -r(a+b-1)-e(\alpha-ar)$ and $\vartheta=\alpha-r(a-1)$.
\end{prop}
\begin{proof}
We apply Beilinson's theorem to retrieve the monad. In order to do so, we compute the cohomology table of $E(-a,-b-1)$.
\begin{center}
\renewcommand\arraystretch{2} 
\begin{tabular}{ccccl}
\multicolumn{1}{c}{$\sO_\xe(-1,-e-1)$}                    & \multicolumn{1}{c}{$\sO_\xe(-1,-e)$}        & \multicolumn{1}{c}{$\sO_\xe(0,-1)$}        & \multicolumn{1}{c}{$\sO_\xe$}            &    \\ \cline{1-4}
\multicolumn{1}{|c|}{0} & \multicolumn{1}{c|}{0}        & \multicolumn{1}{c|}{0}        & \multicolumn{1}{c|}{0}            & $h^3$ \\ \cline{1-4}
\multicolumn{1}{|c|}{$\varepsilon$} & \multicolumn{1}{c|}{$\zeta$}        & \multicolumn{1}{c|}{0}        & \multicolumn{1}{c|}{0}            & $h^2$ \\ \cline{1-4}
\multicolumn{1}{|c|}{0} & \multicolumn{1}{c|}{0} & \multicolumn{1}{c|}{$\eta$}        & \multicolumn{1}{c|}{$\vartheta$}            & $h^1$ \\ \cline{1-4}
\multicolumn{1}{|c|}{0} & \multicolumn{1}{c|}{0}        & \multicolumn{1}{c|}{0} & \multicolumn{1}{c|}{0} & $h^0$ \\ \cline{1-4}
\multicolumn{1}{c}{ $E(-a-1,-b-2)[-1]$}                & \multicolumn{1}{c}{$E(-a,-b-2)[-1]$}  & \multicolumn{1}{c}{$E(-a-1,-b-1)$}  & \multicolumn{1}{c}{$E(-a,-b-1)$}      &   
\end{tabular}
\label{tabella1}
\end{center}
Since $\xe$ is not embedded as a scroll, we obtain all the vanishings in the table with Lemma \ref{cohomologyulrich}. To compute the dimension of the only non-zero cohomology groups we use Riemann-Roch. So we have
\begin{itemize}
\item [$\bullet$]$\varepsilon=-\chi(E(-a-1,-b-2))= 2\alpha+\beta-r(2a+b-1)-e(\alpha-ar)$.
\item [$\bullet$]$\zeta=-\chi(E(-a,-b-2))=2\alpha-2r(a-1)$.
\item [$\bullet$]$\eta=-\chi(E(-a-1,-b-1))=\alpha+\beta -r(a+b-1)-e(\alpha-ar)$.
\item [$\bullet$]$\vartheta=-\chi (E(-a,-b-1))= \alpha-r(a-1)$.
\end{itemize}
By Theorem \ref{Beilinson} and Remark \ref{rembeil} we get that $E(-a,-b-1)$ is the cohomology of the monad
\[
0\rightarrow \sO_{\xe}^\varepsilon(-1,-e-1)\rightarrow \sO_{\xe}^\zeta(-1,-e) \oplus \sO_{\xe}^\eta(0,-1)\rightarrow \sO_{\xe}^\vartheta\rightarrow 0
\]
and tensoring it by $\sO_{\xe}(a,b+1)$ we obtain the desired result.
\end{proof}

\section{From the resolution to the Ulrich bundle}

Now we study the inverse problem: given a coherent sheaf $E$ which is the cokernel (resp. kernel) of a map as in \eqref{risoluzionerul} (resp in \eqref{risoluzionerulkernel}), is it an Ulrich sheaf on $\xe$?

\begin{theorem}\label{inversecondition}
Let $(\xe,\sO_\xe(h))$ with $h=a C_0+bf$ be a polarized Hirzebruch surface. 
\begin{enumerate}
\item Let $\phi$ be an injective map 
\begin{equation}\label{mappainversa}
 \sO_\xe^\gamma (a-1,b-e-1) \overset{\phi}{\longrightarrow} \sO_\xe^\delta (a-1,b-e)\oplus \sO_\xe^\tau(a,b-1)
\end{equation}
with $\delta,\tau$ non-negative, $\gamma$ positive and $\delta+\tau>\gamma$. Let us call $r=\delta+\tau-\gamma$ and denote with $E$ the cokernel of $\phi$. In particular $c_1(E)=\alpha C_0+\beta f$ with $\alpha=\tau+r(a-1)$ and $\beta=r(b-1)+\delta-e(r-\tau)$. If $c_1(E)h=\frac{r}{2} (3h^2+hK_\xe)$ then $E$ is an Ulrich bundle if and only if $H^2(\xe,E(-2h))=0$.
\item If $\psi$ is a surjective map
\begin{equation}
\sO_\xe^\lambda(2a-1,2b-2) \oplus \sO_\xe^\mu (2a-2,2b-1-e)\overset{\psi}{\longrightarrow} \sO_\xe^\nu (2a-1,2b-1)
\end{equation}
with $\lambda,\mu$ non-negative, $\nu$ positive and $\lambda+\mu>\nu$. Let us call $r=\lambda+\mu-\nu$ and denote by $E$ the kernel of $\psi$. In particular $c_1(E)=\alpha C_0+\beta f$ with $\alpha=r(2a-1)-\mu$ and $\beta=r(2b-1-e)-\lambda-e(\mu-r)$. If $c_1(E)h=\frac{r}{2} (3h^2+hK_\xe)$ then $E$ is an Ulrich bundle if and only if $H^0(\xe,E(-h))=0$.
\end{enumerate}\end{theorem}

\begin{proof}
We only prove (1), since the proof of (2) is completely analogous. First of all observe that the existence of an injective map $\phi$ is guaranteed by the fact that $\sO_\xe^{\gamma\delta} (0,1)\oplus \sO_\xe^{\gamma\tau}(1,e)$ is globally generated \cite[\S 4.1]{Ban}.

Let $E$ be the cokernel of $\phi$, thus $E$ fits into an exact sequence like \eqref{risoluzionerul}. So as soon as we check that
\begin{equation*}
H^0(\xe,E(-h))=H^1(\xe,E(-h))=0 \qquad \text{and} \qquad H^1(\xe,E(-2h))=H^2(\xe,E(-2h))=0,
\end{equation*}
then $E$ is an Ulrich vector bundle. Let us consider $E(-h)=E(-a,-b)$. Now tensor \eqref{risoluzionerul} by $\sO_\xe(-a,-b)$. 
Since 
\[
H^i(\xe,\sO_\xe^{\gamma}(-1,-e-1))=H^i(\xe,\sO_\xe^{\delta}(-1,-e)\oplus\sO_\xe^{\tau}(0,-1))=0 \ \text{for all $i$},
\]
we get
\[
h^0(\xe,E(-a,-b))=h^1(\xe,E(-a,-b))=0.
\]
Now we focus on $E(-2h)=E(-2a,-2b)$. We tensor \eqref{risoluzionerul} by $\sO_\xe(-2a,-2b)$. Setting $A=\sO_\xe^{\gamma}(-a-1,-b-e-1)$ and $B=\sO_\xe^{\delta}(-a-1,-b-e)\oplus\sO_\xe^{\tau}(-a,-b-1)$ the induced long exact sequence in cohomology takes the form
\[
0\rightarrow H^1(\xe,E(-2h))\rightarrow H^2(\xe,A)\xrightarrow []{H^2(\phi(-2h))} H^2(\xe,B)\rightarrow H^2(\xe,E(-2h))\rightarrow 0.
\]
We show that $h^2(\xe,A)=h^2(\xe,B)$. This will imply $H^1(\xe,E(-2h))=H^2(\xe,E(-2h))$. By Riemann-Roch we obtain that if $c_1(E)h=\frac{r}{2}(3h^2+hK_\xe)$, then $\chi(A)=\chi(B)$. In particular $\chi(E(-2h))=0$ and since $H^0(\xe,E(-2h))=0$ we have that $$H^1(\xe,E(-2h))=H^2(\xe,E(-2h)).$$ Thus $E=\Coker(\phi)$ is Ulrich if and only if $H^2(\xe,E(-2h))=0$ (which is equivalent to the injectivity of the map $H^2(\phi(-2h))$).

\end{proof}
In \cite{Aprodu} the authors proved the existence of special rank two Ulrich bundles on ruled surfaces. Thanks to Theorem \ref{inversecondition} we obtain the existence of special Ulrich bundles on Hirzebruch surfaces for any very ample polarization in a different way. Furthermore we characterize them as the cokernel (resp. kernel) of an injective (resp. surjective) map between certain totally decomposed vector bundles.
\begin{corollary}\label{esistenzaspeciali}
Let $E$ be a rank two vector bundle on $(\xe,\sO_\xe(h))$ with $h=aC_0+bf$ and $c_1(E)=\alpha  C_0 + \beta f=3h+K_{\xe}$. Then $E$ is Ulrich if and only if it is the cokernel of an injective map $\phi$
\[
\sO_\xe^\gamma (a-1,b-e-1) \overset{\phi}{\longrightarrow} \sO_\xe^\delta (a-1,b-e)\oplus \sO_\xe^\tau(a,b-1)
\]
with $\gamma=\alpha+\beta-2(a+b-1)-e(\alpha-2a)$, $\delta=\beta-2(b-1)-e(\alpha-2a)$ and $\tau=\alpha-2(a-1)$. 
\end{corollary}
\begin{proof}
We showed in Theorem \ref{inversecondition} that as soon as a vector bundle has a resolution of the form \eqref{risoluzionerul} then $h^i(\xe,E(-h))=0$ for all $i$. Furthermore, $E$ is a rank two vector bundle, so $E^\vee \cong E(-c_1)$. From the resolution we are able to conclude that $h^1(\xe,E(-2h))=h^2(\xe,E(-2h))$ and using Serre's duality we get
\[
h^2(\xe,E(-2h))=h^0(\xe,E^\vee(2h+K_{\xe}))=h^0(\xe,E(-h))=0,
\]
thus $E$ is Ulrich. Notice that in the case of special Ulrich bundles the resolutions \eqref{risoluzionerul} and \eqref{risoluzionerulkernel} are dual to each other (up to twist).
\end{proof}
\begin{remark}
In the case of $(\mathbb{P}^2,\sO_{\mathbb{P}^2}(d))$ one can find a resolution of an Ulrich bundle using similar techniques (see \cite{coskun2}, \cite{lin}). In that case, every rank $r$ vector bundle admitting a resolution of the form
\begin{equation}
0\rightarrow{\sO_{\mathbb{P}^2}}^{\frac{r}{2}(d-1)}(d-2)\overset{\phi}{\rightarrow} {\sO_{\mathbb{P}^2}}^{\frac{r}{2}(d+1)}(d-1)\rightarrow E \rightarrow 0
\end{equation}
is Ulrich if and only if $H^2(\mathbb{P}^2,E(-2d))=0$. When we consider rank two vector bundles sitting in the previous short exact sequence, they are automatically Ulrich, i.e. the cohomological condition is trivially satisfied (using the fact that $E^\vee\cong E(-c_1))$. However we will see that the situation on rationally ruled surfaces is quite different.
\end{remark}
In order to show that the vanishing of $H^1(\xe,E(-2h))$ is needed in general, we focus our attention on rank two Ulrich bundles on $(X_0,\sO_{X_0}(d,d))$. We will see in Section 5 that,  thanks to Proposition \ref{admissiblechern}, we have $d+1$ admissible first Chern classes $\alpha  C_0 +\beta f$ (up to an exchange of $\alpha$ and $\beta$) and we have that for $(\alpha,\beta)=(2d-2,4d-2)$ the Ulrich bundle splits since $\Ext^1(\sO_{X_0}(d-1,2d-1),\sO_{X_0}(d-1,2d-1))=0$. 

We start recalling which are the Ulrich line bundles on $(X_0,\sO_{X_0}(a,b))$. 
\begin{prop}\cite[Example 2.1]{casnati4}\label{linebundlesX0}
Let $L$ be a line bundle on the polarized surface $(X_0,\sO_{X_0}(a,b))$. Then $L$ is Ulrich if and only if
\begin{equation*}
L=\sO_{X_0}(2a-1,b-1) \quad \text{or} \quad L=\sO_{X_0}(a-1,2b-1).
\end{equation*}
\end{prop}
Now we are able to construct a counterexample of a vector bundle realized as the cokernel of a map as in Theorem \ref{inversecondition} which is not Ulrich (in particular with $H^1(\xe,E(-2h))\cong H^2(\xe,E(-2h))\neq 0$).
\begin{example} 
Let us consider the polarized Hirzebruch surface $(X_0,\sO_{X_0}(d,d))$ and let $u$ be an integer such that $1\le u \le d-1$. We construct a rank 2 vector bundle sitting in a resolution of the form \eqref{risoluzionerul}, with first Chern class $c_1=(2d-2+u,4d-2 -u)$ that is not Ulrich. Let us consider the following exact sequence

\[
0 \rightarrow \sO_{\p}^{u-1}(d-1)\rightarrow \sO_{\p}^{u}(d)\rightarrow \sO_{\p}(u+d-1)\rightarrow 0
\]
and let us pull it back on $X$ obtaining
\[
0 \rightarrow \sO_{X_0}^{u-1}(d-1,d-1)\rightarrow \sO_{X_0}^{u}(d,d-1)\rightarrow \sO_{X_0}(u+d-1,d-1)\rightarrow 0.
\]
With the same argument we can find a second short exact sequence
\[
0\rightarrow \sO_{X_0}^{2d-u-1}(d-1,d-1)\rightarrow \sO_{X_0}^{2d-u}(d-1,d) \rightarrow \sO_{X_0}(d-1,3d-u-1)\rightarrow 0.
\]
If we set $E=\sO_{X_0}(u+d-1,d-1) \oplus \sO_{X_0}(d-1,3d-u-1)$ and we combine the two sequences we obtain a resolution of the form \eqref{risoluzionerul}:
\[
0\rightarrow \sO_{X_0}^{2d-2}(d-1,d-1) \rightarrow \sO_{X_0}^{2d-u}(d-1,d)\oplus \sO_{X_0}^{u}(d,d-1)\rightarrow E \rightarrow 0.
\]
Every direct summand of an Ulrich bundle is also Ulrich. By Proposition \ref{linebundlesX0}, we know that both $\sO_{X_0}(u+d-1,d-1)$ and $\sO_{X_0}(d-1,3d-1-u)$ are Ulrich only when $u=d$, so the bundle $E$ constructed in this way is not Ulrich.
\end{example}
We conclude this section with the following remark.
\begin{remark}\label{divisorenonUlrich}
In the same hypothesis of Theorem \ref{inversecondition}, the locus of maps $\phi$ which do not give rise to Ulrich bundles is a divisor in the open space of maximal rank matrices which represent morphisms $\phi$. In fact it is the locus where the induced map in cohomology
\[
H^2(\xe,\sO_\xe^\gamma(-a-1,-b-e-1))\xrightarrow[]{H^2(\phi(-2h))} H^2(\xe,\sO_\xe^\delta(-a-1,-b-e) \oplus \sO_\xe^\tau (-a,-b-1))
\]
is not an isomorphism. Since this two vector spaces have the same dimension, the locus where $E=\Coker(\phi)$ is not Ulrich is given by $\det(H^2(\phi(-2h)))=0$. Now we produce an example where we explicitly describe this locus.

Consider $X_0$ embedded with $\sO_{X_0}(2,2)$. By \cite[Theorem 6.7]{casnatianti} there exists a rank two Ulrich bundle $F$ on $(X_0,\sO_{X_0}(2,2))$ with $c_1(F)=3 C_0+5 f$. Consider a rank two vector bundle $E$ with $c_1(E)=3 C_0+5 f$ realized as the cokernel of an injective map
\begin{equation*}
\sO_{X_0}^2(1,1)\overset{\phi}{\longrightarrow} \sO_{X_0}^3(1,2) \oplus \sO_{X_0} (2,1).
\end{equation*}
Now we describe the locus where $E=\Coker (\phi)$ fails to be Ulrich. Recall by Theorem \ref{inversecondition} that E is Ulrich if and only if the induced map in cohomology $H^2(\phi(-2h))$ is injective. By Serre's duality this is equivalent to the surjectivity of a map $\psi: \Hom(\sO_{X_0}^3(-3,-2) \oplus \sO_{X_0} (-2,-3),\sO_{X_0}(-2,-2)) \rightarrow \Hom (\sO_{X_0}^2(-3-3),\sO_{X_0}(-2,-2))$ where $\psi (f)=f \circ \phi$.

Now take 
\[
f \in \Hom(\sO_{X_0}^3(-3,-2)\oplus \sO_{X_0}(-2,-3),\sO_{X_0}(-2,-2)).
\]
Let us denote $[Y_0:Y_1]$ the coordinates of the first factor $\p$ and $[Y_2:Y_3]$ the coordinates of the second factor. Then the matrices of $f$ and $\phi$ are 
\[
f=
\begin{pmatrix}
\alpha_1^0Y_0+\alpha_1^1Y_1 & \alpha_2^0Y_0+\alpha_2^1Y_1 &  \alpha_3^0Y_0+\alpha_3^1Y_1 & \alpha_4^2Y_2+\alpha_4^3Y_3
\end{pmatrix}
\]
and 
\[
\phi=
\begin{pmatrix}
\beta_{1,1}^2Y_2+\beta_{1,1}^3Y_3 & \beta_{1,2}^2Y_2+\beta_{1,2}^3Y_3 \\
\beta_{2,1}^2Y_2+\beta_{2,1}^3Y_3  & \beta_{2,2}^2Y_2+\beta_{2,2}^3Y_3 \\
\beta_{3,1}^2Y_2+\beta_{3,1}^3Y_3  & \beta_{3,2}^2Y_2+\beta_{3,2}^3Y_3 \\
\beta_{4,1}^0Y_0+\beta_{4,1}^1Y_1 & \beta_{4,2}^0Y_0+\beta_{4,2}^1Y_1
\end{pmatrix}.
\]

Imposing $\psi(f)=f\phi=(0)$ we obtain a system 8 equations given by 
\begin{equation}\label{sistemaesempio}
\sum_{l=1}^{3}{\alpha_l^i\beta_{l,k}^j}+\alpha_4^j\beta_{4,k}^i=0 
\end{equation}
where $i=0,1$, $j=2,3$ and $k=1,2$. Now if we consider $\alpha_b^a$ as variables, the matrix of the system \eqref{sistemaesempio} is given by
\begin{equation}
\mathcal{B}=
\begin{pmatrix}
\beta_{1,1}^2 & \beta_{2,1}^2 & \beta_{3,1}^2 & 0 & 0 & 0 & \beta_{4,1}^0 & 0 \\
\beta_{1,2}^2 & \beta_{2,2}^2 & \beta_{3,2}^2 & 0 & 0 & 0 & \beta_{4,2}^0 & 0 \\
\beta_{1,1}^3 & \beta_{2,1}^3 & \beta_{2,2}^3 & 0 & 0 & 0 & 0 & \beta_{4,1}^0 \\
\beta_{1,2}^3 & \beta_{2,2}^3 & \beta_{3,2}^3 & 0 & 0 & 0 & 0 & \beta_{4,2}^0 \\
0 & 0 & 0 & \beta_{1,1}^2 & \beta_{2,1}^2 & \beta_{3,1}^2 & \beta_{4,1}^1 & 0 \\
0 & 0 & 0 & \beta_{1,2}^2 & \beta_{2,2}^2 & \beta_{3,2}^2 &  \beta_{4,2}^1 & 0 \\
0 & 0 & 0 & \beta_{1,1}^3 & \beta_{2,1}^3 & \beta_{2,2}^3 & 0 & \beta_{4,1}^1 \\
0 & 0 & 0 & \beta_{1,2}^3 & \beta_{2,2}^3 & \beta_{3,2}^3 & 0 & \beta_{4,2}^1  
\end{pmatrix}
\end{equation}
and the locus where $E$ is not Ulrich is described by $\det (\mathcal{B})=0$.
\end{remark}
\section{Admissible ranks and Chern classes for Ulrich bundles}
In this section we deal with the admissible first Chern classes and admissible ranks of Ulrich bundles on $(\xe,\sO_\xe(a,b))$. We start with the following definition:

\begin{definition}\label{definitionadmissible}
Let $D=\alpha C_0+\beta f$ be a divisor on $(\xe,\sO_\xe(h))$ with $h=aC_0+bf$ and let $r$ be a positive integer. We say that the pair $(r,D)$ is an admissible Ulrich pair with respect to $h$ if and only if the following conditions hold
\begin{itemize}
\item[$\bullet$] $Dh=\frac{r}{2}(3h^2+hK_\xe)$;
\item[$\bullet$] $\alpha$ and $\beta$ satisfy the numerical conditions 
\begin{gather}
r(a-1)\leq \alpha \leq r(2a-1),\label{alphanumerico}\\ 
r(b-1)-e(\alpha-r(2a-2)) \leq \beta \leq r(2b-1)-e(r(2a-1)-\alpha), \label{betanumerico}
\end{gather}
with strict equalities in \eqref{alphanumerico} if $e>0$ and $a>1$.
\end{itemize}
\end{definition}

We will omit $h$ in the notation when no confusion arises. It follows trivially from the definition that if $(r,D)$ is not an admissible Ulrich pair, then there cannot exist a rank $r$ Ulrich bundles on $(\xe,\sO_\xe(h))$ with $c_1(E)=D$. Furthermore, once the rank is fixed, we will sometimes consider admissible first Chern classes instead of admissible Ulrich pairs.
\begin{remark}\label{specialcentral}
The bounds in Definition \ref{definitionadmissible} are obtained using Corollary 3.3. In particular the bound on $\beta$ would be $r(b-1)-e(\alpha-ar) \leq \beta \leq r(2b-1)-e(r(2a-1)-\alpha)$ but it is possible to improve it. Recall by Proposition \ref{lemmachernruled} that $E$ is Ulrich if and only if $E^\vee(3h+K_\xe)$ is Ulrich, so the bounds on $\beta$ should be centred in $\frac{r}{2}(3b-2-e)$. This gives us the actual bounds on $\beta$ in Definition \ref{definitionadmissible}.
\end{remark}

In the next proposition we characterize admissible Ulrich pairs.

\begin{prop}\label{admpair}
Let $D=\alpha C_0+\beta f$ be a divisor on $(\xe,\sO_\xe(h))$ with $h=aC_0+bf$. If (r,D) is an admissible Ulrich pair then it satisfies
\begin{equation}\label{admissiblenumeric}
\textrm{T}=\frac{er}{2}(3a-1)+\frac{b}{a}(\alpha+r) \in \mathbb{Z} \qquad \text{and} \qquad r(a-1)+\frac{era}{2b}(a-1) \leq \alpha \leq r(2a-1)-\frac{era}{2b}(a-1).
\end{equation}
Conversely any pair $(r,D)$ satisfying \eqref{admissiblenumeric} and $Dh=\frac{r}{2}(3h^2+hK_\xe)$ is an admissible Ulrich pair.
\end{prop}
\begin{proof}
Since $(r,D)$ is an admissible Ulrich pair then $Dh=\alpha(b-ea)+\beta a=\frac{r}{2}(3h^2+hK_\xe)$. 
Now we express $\beta$ as a function of $\alpha$, thus
\[
\beta= r(3b-1)+\alpha e -\frac{er}{2}(3a-1)-\frac{b}{a}(r+\alpha).
\]
However $\alpha$ and $\beta$ represent the coefficients of the divisor $D$, thus they must be integers. In particular we obtain that 
\[
\frac{er}{2}(3a-1)+\frac{b}{a}(\alpha+r) \in \mathbb{Z}.
\]
Moreover observe that since we expressed $\beta$ as a function of $\alpha$, the numerical conditions \eqref{betanumerico} on $\beta$ give us the bound on $\alpha$. In fact by imposing  $\beta \leq r(2b-1)-e(r(2a-1)-\alpha)$ one obtains $$ r(a-1)+\frac{era}{2b}(a-1) \leq \alpha.$$ The upper bound is obtained as in Remark \ref{specialcentral} by noticing that the interval giving the bounds for $\alpha$ should be centered in $\frac{r}{2}(3a-2)$.  For the second part of the statement it is enough to show that the bound \eqref{admissiblenumeric} on $\alpha$ gives us the bound on $\beta$, but this follows from the fact that we obtained \eqref{admissiblenumeric} by imposing the inequality \eqref{betanumerico} on $\beta$.
\end{proof}
Now we focus on rank two Ulrich bundles. Notice that in this case $\frac{er}{2}(3a-1)$ is always even, thus the admissibility of the first Chern class depends on whether $\frac{b}{a}(\alpha+r)$ is an integer or not.
\begin{prop}\label{admissiblechern}
Let $E$ be a rank two Ulrich bundle on $(X_0,\sO_{X_0}(a,b))$. If $\GCD(a,b)=s$ then we have $2s+1$ possible first Chern classes for $E$ given by 
\[
(2a-2+kq) C_0+(4b-2-kp)f
\]
with $0\le k \le 2s$, $p=\frac{b}{s}$ and $q=\frac{a}{s}$.
\end{prop}

\begin{proof}
Suppose that $c_1(E)$ is given by $c_1(E)=\alpha C_0+\beta f$. Then by Proposition \ref{lemmachernruled} we have:
\begin{equation}\label{c1hab}
c_1(E)h=a\beta+b\alpha=6ab-2a-2b
\end{equation}
which is an integer. Now solving with respect to $\beta$ we obtain
\[
\beta= \frac{1}{a}(6ab-2a-2b -b\alpha).
\]
In order for $c_1(E)$ to be an admissible first Chern class of an Ulrich bundle, by Proposition \ref{admpair} we have the following bounds for $\alpha$ and $\beta$
\[
2a-2\leq \alpha \leq 4a-2 \qquad \text{and} \qquad 2b-2\leq \beta \leq 4b-2.
\]
Now define
\[
\begin{sistema}
\alpha_i= 2a-2+i \\
\beta_i=  \frac{1}{a}\bigl(6ab-2a-2b -b(2a-2)-b\alpha_i\bigr)=4b-2-\frac{bi}{a},
\end{sistema}
\]
By Proposition \ref{admpair}, $(2,\alpha_iC_0+\beta_i f)$ is an admissible Ulrich pair if and only if $\frac{bi}{a}$ is an integer. Let us define $s=\GCD(a,b)$. Then the only possibilities for $\frac{bi}{a}$ to be integer are when $i=k\frac{a}{s}$, with $0\le k\le 2s$ and $k \in \mathbb{Z}$. So we have $2s+1$ admissible first Chern classes for $E$. 
\end{proof}

\begin{remark}
In the case of $(X_0,\sO_{X_0}(a,b))$, to satisfy the Bogomolov's inequality for semistable rank two vector bundles is equivalent to satisfy the numerical conditions in Corollary \ref{condizioninumeriche}. In fact for a semistable rank two vector bundle $E$ with first Chern class $c_1(E)=\alpha  C_0 + \beta f$ the Bogomolov's inequality gives us
\[
\Delta = 4c_2(E)-c_1^2(E)\geq 0,
\]
while the numerical conditions for the Beilinson's resolution are
\[
\begin{sistema}
4a-2 \geq \alpha \geq 2a-2 \\
4b-2 \geq \beta \geq 2b-2.
\end{sistema}
\]
Using Proposition \ref{lemmachernruled} we have 
\[
c_2(E)=\frac{{c_1(E)}^2}{2}-\frac{c_1(E)K_{X_0}}{2}-2(h^2-1),
\]
so that 
\[
\Delta = c_1(E)^2 -2c_1(E)K_{X_0}-2(h^2-1).
\]
Recall that for a rank two Ulrich bundle we have
\[
c_1(E)h=\alpha b + \beta a =3h+K_{X_0},
\]
so expressing $\Delta$ in terms of $\alpha$ we have $\Delta (\alpha)\geq 0$ for all $2a-2\le \alpha \le 4a-2$. We conclude that any rank two vector bundle on $X_0$ fitting into the resolution \eqref{risoluzionerul} satisfies Bogomolov's inequality and, conversely, any vector bundle satisfying the Bogomolov's inequality also satisfies the conditions in Corollary \ref{condizioninumeriche}

\end{remark}

For $(\xe,\sO_\xe(a,b))$ with $e>0$ the situation is slightly different from $X_0$, because when the invariant $e$ is positive is not always guaranteed the existence of Ulrich line bundles.

\begin{remark}
Polarizations $h=C_0+bf$ with $b>e$ are the only ones such that $\xe$ admits Ulrich line bundles. In this case it is immediate to see that we have three admissible first Chern classes for a rank two Ulrich bundle, which are 
\[
(4b-2-2e)f \qquad C_0+(3b-2-e)f \qquad 2C_0+(2b-2)f.
\]
\end{remark}
For all the other cases we have the following proposition.
\begin{prop}
Let $E$ be a rank two Ulrich bundle on $(\xe,\sO_\xe(h))$ with $h=a C_0+b f$ such that $a>1$. If $\GCD(a,b)=s$ then the admissible first Chern classes for $E$ are given by
\[
(2a-2 + kq) C_0+(4b-2-e - k p + kqe) f
\]
with $k\in \mathbb{Z}$ such that $\frac{es}{b}(a-1)\leq k \leq 2s-\frac{es}{b}(a-1)$, $p=\frac{b}{s}$ and $q=\frac{a}{s}$.
\end{prop}
\begin{proof}
The strategy is completely analogous to the one of Proposition \ref{admissiblechern} but the computations are a bit more tedious. Since the degree of $c_1(E)$ is fixed for any rank two Ulrich bundle, this will give us an equation in the coefficients of $c_1(E)$, namely $\alpha$ and $\beta$. By solving for $\beta$ and imposing that it is an integer number, we obtain all the admissible first Chern classes.
\end{proof}
\begin{remark}\label{admissrank2spec}
Suppose $\xe$ is embedded with a very ample divisor $h=aC_0+bf$ such that $\GCD(a,b)=1$, then the only possibility for the first Chern class of a stable, rank two Ulrich bundle $E$ on $\xe$ is $c_1(E)=3h+K_{\xe}$ (i.e. $E$ is special).
\end{remark} 
In what follows we focus on the admissibility of ranks and first Chern classes of Ulrich bundles in some particular cases. In light of Remark \ref{admissrank2spec}, one expects that for polarizations $\sO_\xe(a,b)$ with $\GCD(a,b)=1$ we have the least possible number of admissible first Chern classes for $E$. 

Let $E$ be a rank $r$ Ulrich bundle on $\xe$ with respect to $\sO_\xe(a,b)$ with $\GCD(a,b)=1$ and let $c_1(E)=\alpha C_0+\beta f$ be its first Chern class. Suppose $e>0$ and $a>1$. If $a$ is odd or $e$ is even, then by Proposition \ref{admpair} the pair $(r,c_1(E))$ is an admissible Ulrich pair if and only if 
\begin{equation}\label{admpairspecial}
\alpha=ka-r 
\end{equation}
with 
\[
k\in\mathbb{Z} \qquad \text{and} \qquad r+\frac{er}{2b}(a-1) \leq k \leq 2r-\frac{er}{2b}(a-1).
\]
Now we describe explicitly  all the admissible Ulrich pairs for such polarized Hirzebruch surfaces.
\begin{prop}\label{admissiblespecial}
Let $E$ be a rank r Ulrich bundle on $(\xe,\sO_\xe(h))$ with $h=a C_0+b f$ such that $\GCD(a,b)=1$. Suppose $e>0$ and $a>1$. If $a$ is odd or $e$ is even then the admissible first Chern classes for $E$ are given by
\begin{equation}\label{admissibleexplicit}
(ka-r) C_0+\Bigl(r\bigl(3b-1-\frac{e}{2}(3a+1)\bigr)+k(ae-b)\Bigr) f
\end{equation}
with $k\in \mathbb{Z}$ such that $r+\frac{er}{2b}(a-1) \leq k \leq 2r-\frac{er}{2b}(a-1)$.
\end{prop}
\begin{proof}
In the hypothesis of this Proposition $c_1(E)=\alpha C_0+\beta f$ is admissible if and only if $\alpha=ka-r$ with $r+\frac{er}{2b}(a-1) \leq k \leq 2r-\frac{er}{2b}(a-1)$. For each of such $\alpha$ use the relation $$\alpha(b-ae)+\beta a=\frac{r}{2}(3h^2+hK_\xe)$$ to compute $\beta$.
\end{proof}


\section{Existence and moduli spaces}
In this section we will discuss some results on the existence of Ulrich bundles and their moduli spaces. Let us fix the notation. Let $X$ be a smooth algebraic surface, we will denote by $M_{h}(r;c_1,c_2)$ the moduli space of rank two locally free sheaves $E$ on $X$ stable with respect to a polarization $h$ and with $\det(E)=c_1 \in \Pic(X)$ and $c_2(E)=c_2 \in \mathbb{Z}$. We start by recalling the following proposition concerning the rank two case.
\begin{prop}\cite[Theorem 4.7]{costa2}\label{modulirk2}
Let $X$ be a smooth, irreducible, projective, minimal, rational surface, $c_1\in \Pic (X)$ and $c_2 \in \mathbb{Z}$. Then, for any polarization $h$ on $X$, the moduli space $M_h(2;c_1,c_2)$ is a smooth, irreducible, rational, quasi-projective variety of dimension $4c_2-c_1^2-3$, whenever non-empty.
\end{prop}
The moduli space of stable rank $r$ Ulrich bundles $E$ with $\det (E)=c_1$ and $c_2(E)=c_2$ is an open subset in $M_h(r;c_1,c_2)$ and we will denote it by $M_h^U(r;c_1,c_2)$. In what follows we show how we can use Theorem \ref{mainth} to study the moduli spaces of Ulrich bundles of rank greater than two. We start by giving an existence theorem.

We showed in the previous section that given an injective map $\phi$ as in Theorem \ref{inversecondition}, in general the vanishing $H^1(\xe,E(-2h))=0$ is necessary to obtain that $\Coker(\phi)$ is Ulrich. However we are able to prove the following existence result. 
\begin{theorem}\label{esistenzaulrich}
$(\xe,\sO_\xe(h))$ supports Ulrich bundle of any admissible Ulrich pair $(r,c_1)$. 
\end{theorem}
\begin{proof}
First of all recall that by Proposition \ref{linebundlesruled} and \ref{linebundlesX0}, $(\xe,\sO_\xe(h))$ admits Ulrich line bundles if and only if $e=0$ or $h=C_0+bf$ and $e>0$. Now let us consider a map $\phi$, as in \eqref{mappainversa}, general. Then $\phi$ would be injective and let us denote by $E$ its cokernel. Corollary \ref{condizioninumeriche} implies that the Chern character of $E(-2h)$ satisfies the hypothesis of \cite[Theorem 1.1]{hirzebruch}. In particular we obtain that $E(-2h)$ has natural cohomology and, since $\chi (E(-2h))=0$, we have $h^1(\xe,E(-2h))=h^2(\xe,E(-2h))=0$ and $E$ is Ulrich.
\end{proof}
Once the existence is settled we focus on moduli spaces of Ulrich bundles. 
\begin{lemma}\label{smoothness}
If $E$ is a rank $r$ Ulrich bundle on $(\xe,\sO_\xe(h))$ with $h=aC_0+bf$ and $a,b>1$, then $\Ext^2(E,E)=0$.
\end{lemma}
\begin{proof}
Consider the short exact sequence \eqref{risoluzionerul} and tensor it by $E^\vee$. The long exact sequence in cohomology gives us
\[
\delta h^2(\xe,E^\vee(a-1,b-e))+\tau h^2(\xe,E^\vee(a,b-1))\geq h^2(\xe,E\otimes E^\vee).
\]
Since $E$ is Ulrich, the same is true for $E^\vee(3h+K_\xe)$. Now using Lemma \ref{cohomologyulrich} and the fact that $a,b > 1$, we obtain $h^2(\xe,E^\vee(a-1,b-e))=h^2(\xe,E^\vee(a,b-1))=0$, thus $H^2(\xe,E\otimes E^\vee )\cong \Ext^2(E,E)=0$.
\end{proof}
In light of this we can state the following proposition.
\begin{prop}\label{modspaces}
Let us consider $(\xe,\sO_\xe(h))$ with $h=aC_0+bf$ and let $(r,c_1)$ be an admissible Ulrich pair. Then the moduli space $M_h^U(r;c_1,c_2)$ is a smooth, irreducible, unirational, quasi-projective variety of dimension $c_1^2-rc_1K_\xe-r^2(2h^2-1)+1$, for $r=2,3$, $e>0$ and $a>1$.
\end{prop}
\begin{proof}
Smoothness comes from Lemma \ref{smoothness}. To any element 
\[
\phi \in \Hom(\sO_\xe^\gamma(a-1,b-e-1),\sO_\xe^\delta(a-1,b-e)\oplus \sO_\xe^\tau(a,b-1))
\]
as in Theorem \ref{inversecondition}, we can associate its cokernel, forming a flat family. Thanks to Theorem \ref{esistenzaulrich} the generic element in this family is an Ulrich bundle. Theorem \ref{ulrichch} tells us that an Ulrich bundle can only be destabilized over an Ulrich bundle and the cokernel of the inclusion map is also Ulrich. In particular, if we consider the rank $r$ to be two or three, the existence of a strictly semistable Ulrich bundle would imply the existence of Ulrich line bundles. Using Proposition \ref{linebundlesruled}, we see that $(\xe,\sO_\xe(h))$ does not admit Ulrich line bundles for $e>0$ and $a>1$. We conclude that in these cases all the Ulrich bundles are stable, so $M_h^U(r;c_1,c_2)$ is non-empty. For irreducibility and unirationality observe that $M_h^U(r;c_1,c_2)$ is dominated by an open subset of a space of matrices, which is irreducible and unirational. Finally, for the dimension, recall that for stable bundles we have $h^0(E\otimes E^\vee)=1$. Since $\dim M_h^U(r;c_1,c_2)=\dim \Ext^1(E,E)$, and $\Ext^2(E,E)=0$ by Lemma \ref{smoothness}, we obtain $\dim M_h^U(r;c_1,c_2)=1-\chi(E\otimes E^\vee)$. Using \cite[Proposition 2.12]{cashart} we have the desired result.
\end{proof}

\begin{remark} 
When $(\xe,\sO_\xe(a,b))$ admits Ulrich line bundles the situation is different. In \cite{malaspina2} the authors proved that when $a=1$ we have exactly two Ulrich line bundles and all the Ulrich bundles of rank greater than two are strictly semistable, i.e $M_h^U(r;c_1,c_2)$ is empty for $r\geq 2$. If we consider $(X_0,\sO_{X_0}(a,b))$ with $a,b> 1$, by Proposition \ref{linebundlesX0} it always admits two Ulrich line bundles. However we will see in Section 7 that there exists a stable rank two Ulrich bundle for every admissible first Chern class. Thus we can describe the moduli space $M_{X_0,h}^U(2;c_1,c_2)$ using the same argument of Proposition \ref{modspaces}.
\end{remark}

We continue this section dealing with higher ranks Ulrich bundles. Suppose $\xe$ is embedded with a very ample line bundle $\sO_\xe(h)$ with $h=a C_0+bf$ and such that $\GCD(a,b)=1$, i.e. the embedding $\xe \hookrightarrow \mathbb{P}^N$ does not factor through a Veronese embedding, and with $a>1$ and odd. Once we fix the rank $r$  the admissible Ulrich pairs are given by \eqref{admissibleexplicit}. 

\begin{remark}
It is worth to notice that in Proposition \ref{admissiblespecial} we have $r+\frac{er}{2b}(a-1) < 2r-\frac{er}{2b}(a-1)$. However $k$ is an integer and it could happen that there is no integer between $r+\frac{er}{2b}(a-1)$ and $2r-\frac{er}{2b}(a-1)$, i.e. in that cases there are no admissible Ulrich pairs $(r,D)$ with respect to $h=aC_0+\beta f$.
\end{remark}

\begin{prop}\label{ammissibilespeciale}
Let us consider $(\xe,\sO(a,b))$ with $\GCD(a,b)=1$, $e>0$ and $a>1$. If $a$ is odd or $e$ is even then for any even rank $r$ the pair $(r,\frac{r}{2}(3h+K_\xe))$ is an admissible Ulrich pair.
\end{prop}
\begin{proof}
Let us set $D=\alpha C_0+\beta f=\frac{r}{2}(3h+K_\xe)$. In particular $\alpha=\frac{3}{2}ra-r$ so by Proposition \ref{admissiblespecial} it is enough to show that $r+\frac{er}{2b}(a-1) \leq \frac{3}{2}r \leq 2r-\frac{er}{2b}(a-1)$. This is equivalent to $r(b-ea+e)\geq0$ which is always true by the very ampleness of $h$.
\end{proof}

In the cases when $r$ is odd then the divisor $\frac{r}{2}(3h+K_\xe)$ does not have integer coefficients, thus the ``nearest" admissible Ulrich pair would be $\frac{r}{2}(3h+K_\xe)-\frac{1}{2}D$ where $D=aC_0+(ae-b)f$. Let us denote by $\Delta_t$ the positive number $$\Delta_t=\frac{tb}{b-ae+e}.$$ Observe that if $r$ is odd and $r<\Delta_1$, then there are no integers between $r+\frac{er}{2b}(a-1)$ and $2r-\frac{er}{2b}(a-1)$, thus we have the following proposition.
\begin{prop}\label{oddrank}
Let us consider $(\xe,\sO(a,b))$ with $\GCD(a,b)=1$, $e>0$ and $a>1$. If $a$ is odd or $e$ is even then there are no admissible Ulrich pair $(r,D)$ with $r$ odd and $r<\Delta_1$, i.e. there cannot exist odd rank Ulrich bundles of rank $r<\Delta_1$.
\end{prop}

\begin{remark}
In the same setting of Proposition \ref{oddrank}, let us consider $\bar{r}$ to be the first odd integer such that $\bar{r}\geq \Delta_1$. Then there exist two admissible Ulrich pairs $\left(\bar{r},\frac{\bar{r}}{2}(3h+K_\xe)-\frac{1}{2}D\right)$ and $\left(\bar{r},\frac{\bar{r}}{2}(3h+K_\xe)+\frac{1}{2}D\right)$ with $D=aC_0+(ae-b)f$. By Proposition \ref{esistenzaulrich} there exists an Ulrich bundle corresponding to these admissible Ulrich pairs. Observe that such a bundle is stable. In fact if it were semistable, then would be an extension of an odd and an even Ulrich bundle with rank smaller than $\bar{r}$, but this is not possible since there are no Ulrich bundles of odd rank smaller than $\bar{r}$.
\end{remark}
Now we prove a Lemma which will be useful in the next propositions
\begin{lemma}\label{lemmacalcoli}
Let $E_i$ be Ulrich bundles on $(\xe,\sO_\xe(h))$ such that $\rango(E_i)=r_i$ are even. Then the admissible first Chern classes for $E_i$ are
\begin{equation}\label{alternativeadmissible}
c_{1}(E_i)=\frac{r_i}{2}(3h+K_\xe)+{k_i}D \quad \text{with} \quad -\left(\frac{r_i}{2}-\frac{er}{2b}(a-1)\right)< k_i < \frac{r_i}{2}-\frac{er}{2b}(a-1)
\end{equation}
with $k_i \in \mathbb{Z}$ and $D=aC_0+(ae-b)f$. In particular 
\begin{equation*}
\chi(E_{i}\otimes E_{j}^\vee)=-\frac{r_ir_j}{4}(h^2-4)+K_\xe D\left(\frac{r_i}{2}k_j-\frac{r_j}{2}k_i\right)+k_ik_jh^2.
\end{equation*}
\end{lemma}
\begin{proof}
The first part of the Lemma is a direct consequence of Proposition \ref{admissiblespecial}. For the computation of the Euler characteristics see \cite[Proposition 2.12]{cashart} using the relations $D^2=-h^2$ and $Dh=0$.
\end{proof}

\begin{prop}\label{esistenzastabilispeciali}
Let us consider $(\xe,\sO(a,b))$ with $\GCD(a,b)=1$, $e>0$ and $a>1$. Suppose $a$ is odd or $e$ is even. Then for any even rank $r<\Delta_1$ there exists a stable rank $r$ Ulrich bundle $E$ with $c_1(E)=\frac{r}{2}(3h+K_\xe)$.
\end{prop}
\begin{proof}
First of all observe that if $\Delta_1 \leq 2$ then there are no $r$ satisfying the hypothesis, thus let us suppose $\Delta_1\geq 3$. We will use the same idea of a method that M. Casanellas and R. Hartshorne used in \cite[Theorem 4.3]{cashart}. To show the existence of rank $2t<\Delta_1$ simple Ulrich bundles with $c_1(E)=t(3h+K_\xe)$, we proceed by induction on half the rank $t$. The existence of stable special rank two Ulrich bundles is given by Proposition \ref{modspaces} and \ref{ammissibilespeciale}. Now suppose that for any $s<t$ there exists a rank $2s$ stable Ulrich bundle with first Chern class equal to $s(3h+K_\xe)$. By inductive hypothesis there exist stable Ulrich bundles $F$ and $G$ of ranks $2$ and $2t-2$ respectively, such that $c_1(F)+c_1(G)=t(3h+K_\xe)$. Now consider a non-split extension
\[
0 \to F \to E \to G \to 0.
\]
The bundle $E$ is a simple Ulrich bundle of rank $2t$ (see \cite[Lemma 4.2]{cashart}). Notice that this is possible since $\dim \Ext^1(G,F)=h^1(\xe,F\otimes G^\vee)>0$ by Lemma \ref{lemmacalcoli}. Now consider the modular family of simple bundles $E$. We can compute its dimension using Lemma \ref{lemmacalcoli} as $h^1(\xe,E\otimes E^\vee)=t^2(h^2-4)+1$. Now we show that the dimension of each family of strictly semistable Ulrich bundles of rank $2t$ which are obtained as an extension of two stable Ulrich bundles is strictly smaller than $h^1(\xe,E\otimes E^\vee)$. Consider two stable Ulrich bundles $F_1$ and $F_2$ of rank $2t_1$ and $2t_2$ respectively, such that $t_1+t_2=t$. We have $c_1(F_j)=t_j(3h+K_\xe)$. Now we show that 
\begin{equation}\label{dimensionalcount}
\dim \{F_1\}+\dim\{F_2\} +\dim(\Ext^1(F_2,F_1))-1<h^1(\xe,E\otimes E^\vee),
\end{equation}
i.e. we want to show that
\[
(t_1^2+t_2^2+t_1t_2)(h^2-4)<(t_1+t_2)^2(h^2-4)
\]
which is equivalent to
\[
t_1t_2(h^2-4)>0.
\]
Since we supposed $e>0$ and $a>1$, we have that $h^2\geq 12$, thus $t_1t_2(h^2-4)>0$. In particular, we have that the general element in the modular family of simple Ulrich bundles of rank $2(t_1+t_2)=2t$ and $c_1=t(3h+K_\xe)$ is stable.
\end{proof}

We continue with some remarks.

\begin{remark}
In the proof of Proposition \ref{esistenzastabilispeciali} it is enough to consider strictly semistable Ulrich bundles $E$ which are extensions of stable Ulrich bundles. Indeed suppose $E$ is a strictly semistable Ulrich bundle. Then each term of his Jordan-H\"older filtration is a stable Ulrich bundle \cite[Lemma 2.15]{CKM}. Let $F$ be one of them and consider the quotient $F_1=E/F$. Observe that $F_1$ is Ulrich by Proposition \ref{ulrichch}. In this way we can always assume that a strictly semistable Ulrich bundle is an extension of a stable bundle $F$ and a semistable Ulrich bundle $F_1$. Now, if $F_1$ is extension of two stable Ulrich bundles, then using the same dimensional count as in the proof of Proposition \ref{esistenzastabilispeciali}, the family parametrizing $F_1$ has dimension strictly smaller than the family of simple Ulrich bundles with the same invariants as $F_1$. If $F_1$ is an extension of a stable Ulrich bundle and a strictly semistable Ulrich bundle then we iterate this process until we obtain an Ulrich bundle $F_l$ which is extension of two stable Ulrich bundles. However at each step of this process we have an inequality as in \eqref{dimensionalcount}, thus in the end $h^1(\xe,E\otimes E^\vee)$ is strictly greater than the dimension of any family parametrizing strictly semistable Ulrich bundles with the same invariants as $E$.
\end{remark}

\begin{remark}
The hypothesis $r<\Delta_1$ in Proposition \ref{esistenzastabilispeciali} is necessary to exclude the existence of rank odd Ulrich bundles. The inequality $\Delta_{j-1}\leq r<\Delta_j$ gives us information about the admissible Ulrich pairs $(r,D)$. We saw that if $r<\Delta_1$ there are no admissible Ulrich pairs $(r,D)$ with $r$ odd. Using Proposition \ref{ammissibilespeciale} it is possible to see that if $\Delta_1 \leq r < \Delta_2$ then if $r$ is even the only admissible Ulrich pair is given by $\left(r,\frac{r}{2}(3h+K_\xe)\right)$ and if $r$ is odd then we have exactly two admissible Ulrich pairs given by $\left(r,\frac{r}{2}(3h+K_\xe)-\frac{1}{2}D\right)$ and $\left(r,\frac{r}{2}(3h+K_\xe)+\frac{1}{2}D\right)$, with $D=aC_0+(ae-b)f$.

In general given $t\in \mathbb{Z}_{\geq 0}$, if $\Delta_{t-1} \leq r < \Delta_{t}$ then
\begin{itemize}
\item[$\bullet$] there exist $n=2 \left\lfloor \frac{t}{2} \right\rfloor$ admissible Ulrich pairs $(r,D)$ with $r$ odd;
{\vspace{0.1cm}}
\item[$\bullet$] there exist $m=2 \left\lceil \frac{t}{2}\right\rceil-1$ admissible Ulrich pairs $(r,D)$ with $r$ even.
\end{itemize}
Observe that $r<\Delta_r$ for each $r$, thus the maximum number of admissible Ulrich pairs $(r,D)$ is $r-1$. Moreover it is worth to notice that, without the hypothesis $r<\Delta_1$, it is considerably more difficult to use the same strategy of Proposition \ref{esistenzastabilispeciali} to prove the existence of stable Ulrich bundles. In fact, in these cases, an Ulrich bundle can be realized as an extension of two Ulrich bundles in several different ways.
\end{remark}

In the remaining part of this section we compare the results we obtained with the existing literature.
\begin{remark}
In \cite{walter} it has been proved that for any birationally ruled surface $S$ endowed with an ample divisor $h$, the moduli spaces $\overline{M}_h(r;c_1,c_2)$ of semistable vector bundles, whenever non-empty, are irreducible and normal. In addition, the open subspace of stable vector bundles ${M}_h(r;c_1,c_2)$ is smooth. Theorem \ref{esistenzaulrich}, Proposition \ref{modspaces} and \ref{esistenzastabilispeciali} extend these results in the case of Hirzebruch surfaces showing the unirationality of an open subset and the non-emptiness for some  admissible Ulrich pairs $(r,c_1)$ and polarizations.

In \cite{liqin} the authors proved the unirationality, smoothness, irreducibility and non-emptiness of the moduli spaces $M_h(3;c_1,c_2)$ of rank three stable vector bundles on polarized Hirzebruch surfaces for some Chern classes. Thanks to Proposition \ref{modspaces} we partially extend this result for $r=3$, $e>0$, $a>1$ and all the admissible Chern classes of an Ulrich bundle.
\end{remark}



\section{Indecomposable rank two Ulrich bundles}
In this section we will construct rank two stable Ulrich bundles on $\xe$  with respect to a very ample polarization $a C_0+ bf$. Using Serre's correspondence on surfaces, we will construct stable Ulrich bundles on $X_0$ for two of the admissible first Chern classes. Then we will show how to use {\em{Macaulay2}} to produce examples of Ulrich bundle on $X_e$ for several different polarization, Chern classes and ranks.

\begin{prop}
Let us consider $(X_0,\sO_{X_0}(h))$ with $h=aC_0+bf$ and $\GCD(a,b)=s>1$. Then there exists non-special rank two Ulrich bundles with $c_1(E)=(3a-2-\frac{a}{s})C_0+( 3b - 2+ \frac{b}{s})f$ given by 
\begin{equation}\label{nonspecialiCB}
0 \rightarrow \sO_{X_0}(a-1,b+\frac{b}{s}-1)\rightarrow E \rightarrow \mathcal{I}_Z(2a-1-\frac{a}{s},2b-1)\rightarrow 0,
\end{equation}
with $Z$ a general zero dimensional subscheme of $X_0$ with $l(Z)=ab(\frac{s-1}{s})$.
\end{prop}
\begin{proof}
First we prove that there exists vector bundles realized as an extension \eqref{nonspecialiCB}. In order to do so we need to verify that the pair $((a-2-\frac{a}{s}) C_0+(b-\frac{b}{s}-2)f,Z)$ has the Cayley-Bacharach property. We have 
\begin{equation*}
h^0(X_0,\sO_{X_0}(a-2-\frac{a}{s},b-\frac{b}{s}-2))=ab(1-\frac{1}{s})^2-(a+b)(1-\frac{1}{s})+1.
\end{equation*}
An easy computation shows that
\[
h^0(X_0,\sO_{X_0}(a-2-\frac{a}{s},b-\frac{b}{s}-2))\le l(Z)-1.
\]
It follows that for a general $Z$, the pair $\bigl((a-2-\frac{a}{s}) C_0+(b-\frac{b}{s}-2)f,Z\bigr)$ verifies the Cayley-Bacharach property, so in any extension of type \eqref{nonspecialiCB} there are rank two vector bundles. 

By Proposition \ref{lemmachernruled}, in order for $E$ to be Ulrich we need to verify the equalities on the Chern classes and the vanishings in cohomology. Every vector bundle in the extension \eqref{nonspecialiCB} has first Chern class $c_1(E)=3h+K_{X_0}+D$ where $D=(-\frac{a}{s}) C_0+(\frac{b}{s})f$, so we have $c_1(E)h=3h^2+K_{X_0}h$. Furthermore a direct computation shows that 
\begin{align*}
c_2(E)&=l(Z)+\bigl((a-\frac{a}{s}) C_0+(b-\frac{b}{s})f\bigr)\bigl((a-1) C_0+(b+\frac{b}{s}-1)f\bigr)+\\
&+\bigl((a-1) C_0+(b+\frac{b}{s}-1)f\bigr)^2\\
&= \frac{1}{2}(c_1(E)^2-c_1(E)K_{X_0})+r(h^2-1).
\end{align*}
So it remains to check that $h^0(X_0,E(-h))=h^0(X_0,E^\vee(2h+K_{X_0}))=0$. Twisting \eqref{nonspecialiCB} by $\sO_{X_0}(-h)$ and considering the long exact sequence in cohomology we obtain:
\[
h^0(X_0,E(-h))=h^0(X_0,\mathcal{I}_Z(a-1-\frac{a}{s},b-1)).
\]
Furthermore 
\[
h^0(X_0,\sO_{X_0}(a-1-\frac{a}{s},b-1))=l(Z)
\]
thus for a general $Z$ we have $h^0(X_0,\mathcal{I}_Z(a-1-\frac{a}{s},b-1))=0$ and $h^0(X_0,E(-h))=0$. For the second vanishing recall that $E^\vee(2h+K_{X_0})\cong E(-h-D) =E(-a+\frac{a}{s},-b-\frac{b}{s})$. Now tensoring \eqref{nonspecialiCB} by $\sO_{X_0}(-a+\frac{a}{s},-b-\frac{b}{s})$ and considering the long exact sequence in cohomology we get
\[
h^0(X_0,E(-h-D))=h^0(X_0,\mathcal{I}_Z(a-1,b-\frac{b}{s}-1)).
\]
Furthermore
\[
h^0(X_0,\sO_{X_0}(a-1,b-\frac{b}{s}-1))=l(Z),
\]
thus for a general $Z$ we have $h^0(X_0,\mathcal{I}_Z(a-1,b-\frac{b}{s}-1))=0$ and $h^0(X_0,E(-h-D))=0$, so by Proposition \ref{lemmachernruled}, $E$ is a rank two Ulrich bundle on $X_0$.
\end{proof}
\begin{remark}
Recall by Proposition \ref{linebundlesX0} that the only two Ulrich line bundles on $(X_0,\sO_{X_0}(a,b))$ are  $L=\sO_{X_0}(2a-1,b-1)$ and $M=\sO_{X_0}(a-1,2b-1)$. We conclude that a non-special rank two Ulrich bundle, apart from the trivial ones $E=L^{ 2}$ and $E=M^{ 2}$, is always stable. In fact we can check the (semi)stability of a sheaf $E$ by considering the subsheaves $F$ such that the quotient $E/F$ is torsion free (see \cite[Theorem 1.2.2]{oko}). A rank two Ulrich bundle can only destabilize on an Ulrich line bundle, but in this way the quotient would also be an Ulrich line bundle thanks to Theorem \ref{ulrichch}. However, this is not possible because of the numerical conditions imposed by the first Chern classes.
\end{remark}
\begin{remark}
In a completely similar way it is possible to construct rank two special Ulrich bundles for any very ample polarization $h=a C_0+bf$ as extensions
\begin{equation}\label{specialiCB}
0\rightarrow \sO_{X_0}(a-1,b+\frac{b}{s}-1)\rightarrow E \rightarrow \mathcal{I}_Z(2a-1,2b-\frac{b}{s}-1)\rightarrow 0,
\end{equation}
with $Z$ a general zero dimensional subscheme of $X_0$ with $l(Z)=ab(\frac{s-1}{s})$. Although examples of special rank two Ulrich bundles on $X_0$ had already been given by extensions of the two Ulrich line bundles $L=\sO_{X_0}(2a-1,b-1)$ and $M=\sO_{X_0}(a-1,2b-1)$, the Ulrich bundles constructed as in \eqref{specialiCB} are stable. In fact if we tensor \eqref{specialiCB} by $L^\vee$ and consider the long exact sequence in cohomology we get
\[
0\rightarrow H^0(X_0,E\otimes L^\vee) \rightarrow H^0(X_0,\mathcal{I}_Z (0,b-\frac{b}{s})),
\]
but $h^0(X_0,\sO_{X_0}(0,b-\frac{b}{s}))=b-\frac{b}{s}+1\le l(Z)$ so it is not possible to have non-zero maps between $L$ and $E$. Similarly, tensoring \eqref{specialiCB} by $M^\vee$ and taking the induced sequence in cohomology we have
\[
0\rightarrow H^0(X_0,E\otimes M^\vee)\rightarrow H^0(X_0,\mathcal{I}_Z(a,-\frac{b}{s})).
\]
but $H^0(X_0,\mathcal{I}_Z(a,-\frac{b}{s}))=0$, so does not exist a non-zero map between $M$ and $E$. Since a rank two Ulrich bundle can only destabilize on an Ulrich line bundle, $E$ is stable, thus indecomposable.
\end{remark}
Now we produce an alternative description of a rank two non-special Ulrich bundle on $(X_0,\sO_{X_0}(d,d))$ as a non-trivial extension of two non-Ulrich line bundles.
\begin{prop}\label{estensione}
Let $E$ be a rank two Ulrich bundle on $\bigr(X_0,\sO_{X_0}(d,d)\bigl)$ with $c_1(E)=(4d-3) C_0+(2d-1)f$, then $E$ can be represented by an element $\xi \in \Ext^1(\sO_{X_0}(2d-2,2d-1),\sO_{X_0}(2d-1,0))$ i.e. there exists a short exact sequence
\begin{equation}
0 \longrightarrow \sO_{X_0}(2d-1,0) \longrightarrow E \longrightarrow \sO_{X_0}(2d-2,2d-1) \longrightarrow 0.
\end{equation}
Conversely,  if $E$ is a rank two vector bundle corresponding to $\xi \in \Ext^1(\sO_{X_0}(2d-2,2d-1),\sO_{X_0}(2d-1,0))$ then $E$ is Ulrich if and only if it is initialized.
\end{prop}
\begin{proof}
Let us build the Beilinson's table of $E(-2d+1,-2d+1)$.
{\footnotesize
\begin{center}
\renewcommand\arraystretch{2} 
\begin{tabular}{ccccl}
\multicolumn{1}{c}{$\sO_{X_0}(-1,-1)$}                    & \multicolumn{1}{c}{$\sO_{X_0}(-1,0)$}        & \multicolumn{1}{c}{$\sO_{X_0}(0,-1)$}        & \multicolumn{1}{c}{$\sO_{X_0}$}            &    \\ \cline{1-4}
\multicolumn{1}{|c|}{0} & \multicolumn{1}{c|}{0}        & \multicolumn{1}{c|}{0}        & \multicolumn{1}{c|}{0}            & $h^3$ \\ \cline{1-4}
\multicolumn{1}{|c|}{0} & \multicolumn{1}{c|}{$1$}        & \multicolumn{1}{c|}{0}        & \multicolumn{1}{c|}{0}            & $h^2$ \\ \cline{1-4}
\multicolumn{1}{|c|}{0} & \multicolumn{1}{c|}{0} & \multicolumn{1}{c|}{$2d-1$}        & \multicolumn{1}{c|}{2d-2}            & $h^1$ \\ \cline{1-4}
\multicolumn{1}{|c|}{0} & \multicolumn{1}{c|}{0}        & \multicolumn{1}{c|}{0} & \multicolumn{1}{c|}{0} & $h^0$ \\ \cline{1-4}
\multicolumn{1}{c}{ $E(-2d,-2d)[-1]$}                & \multicolumn{1}{c}{$E(-2d,-2d+1)[-1]$}  & \multicolumn{1}{c}{$E(-2d+1,-2d)$}  & \multicolumn{1}{c}{$E(-2d+1,-2d+1)$}      &   
\end{tabular}
\label{tabella1}
\end{center}
}
Observe that the zeroes in the table are obtained using Lemma \ref{cohomologyulrich}.
So in order to compute the numbers in the cohomology table we use the Riemann-Roch theorem. Thus we have
\begin{itemize}
\item[$\bullet$] $\chi(E(-2d+1,-2d+1))=-h^1(X_0,E(-2d+1,-2d+1))=-2(d-1)$.
\item[$\bullet$] $\chi(E(-2d+1,-2d))=-h^1(X_0,E(-2d+1,-2d))=1-2d$.
\item[$\bullet$] $\chi(E(-2d,-2d+1))=-h^1(X_0,E(-2d,-2d+1))=-1$.
\end{itemize}
The first page of the Beilinson's spectral sequence will give us
\begin{equation}\label{p1}
0 \rightarrow \Ker \phi \rightarrow \sO_{X_0}(0,-1)^{2d-1} \overset{\phi}{\rightarrow} \sO_{X_0}^{2d-2} \rightarrow \Coker \phi \rightarrow 0,
\end{equation}
and looking at the second (and infinity) page we have
\begin{equation}\label{p2}
0 \rightarrow \Ker \psi \rightarrow \sO_{X_0}(-1,0) \overset{\psi}{\rightarrow} \Coker \phi \rightarrow 0.
\end{equation}
So in the end we obtain $E(-2d+1,-2d+1)$ as an extension
\begin{equation}\label{p3}
0 \rightarrow \Ker \phi \rightarrow E(-2d+1,-2d+1) \rightarrow \Ker \psi \rightarrow 0.
\end{equation}
Observe that $\Ker \phi$ is locally free since $\phi$ is the pull-back of a map on $\p$ and the kernel of a map on a smooth curve between locally free sheaves is locally free. Furthermore, by \eqref{p3} $\Ker \phi $ can have rank at most 2. We say that $\Ker \phi $ has rank 1. In fact the rank cannot be zero because $\phi$ in \eqref{p1} cannot be injective and the rank cannot be 2 because in that case $\Ker \psi$ would be a torsion sheaf which is in contradiction with \eqref{p2}.

So $\Ker \phi = \sO_{X_0}(0,x)$. Consider \eqref{p2}. Since $E$ is Ulrich we have that $E^\vee (3d-2,3d-2)$ is also Ulrich. In particular $h^i(X_0,E^\vee (2d-2,2d-2))=0$ for all $i$. But $E$ is a rank two vector bundle, so $E\cong E^\vee (c_1)$ and we have the following short exact sequence
\begin{equation}\label{p4}
0 \rightarrow \Ker \phi \otimes \sO_{X_0}(0,2d-2) \rightarrow E^\vee(2d-2,2d-2) \rightarrow \Ker \psi \otimes \sO_{X_0}(0,2d-2) \rightarrow 0.
\end{equation}
Now if we tensor \eqref{p2} by $\sO_{X_0}(0,2d-2)$ we get $h^0(X_0,\Ker \psi \otimes \sO_{X_0}(0,2d-2))=0$ so, considering the long exact sequence in cohomology induced by \eqref{p4} we have
\[
h^0(X_0,\Ker \phi \otimes \sO_{X_0}(0,2d-2))=h^1(X_0,\Ker \phi \otimes \sO_{X_0}(0,2d-2))=0
\]
and the only possibility is to have $x+2d-2=-1$, i.e. $\Ker \phi = \sO_{X_0}(0,1-2d)$.

Now we deal with $\Coker \phi$. Consider \eqref{p2}, then the only two possibilities for $\Ker \psi $ are
\begin{itemize}
\item [ \ i)] $\Ker \psi = \mathcal{I}_Z(0,-1)$ with $Z$ a non-empty zero dimensional subscheme of $X_0$.
\item [\ ii)] $\Ker \psi = \sO_{X_0}(-D)$ with $D$ an effective divisor on $X_0$.
\end{itemize}
If $i)$ holds then we have $c_2(\Coker \phi)=-l(Z)$. But using \eqref{p1} we observe that $c_2(\Coker \phi)=0$ which is in contradiction with $Z$ being non-empty. So it must be $\Ker \psi = \sO_{X_0}(-D)$. Since $c_1(\Coker \phi )=0$ then the only possibility is to have $\sO(-D) = \sO_{X_0}(0,-1)$ and $\Coker \phi =0$. In this way we obtain
\[
0 \rightarrow \sO_{X_0}(0,1-2d) \rightarrow E(-2d+1,-2d+1) \rightarrow \sO_{X_0}(-1,0) \rightarrow 0.
\]
and tensoring it by $\sO_{X_0}(2d-1,2d-1)$ we obtain the desired result.

Conversely take an extension
\begin{equation}\label{ext2}
0 \rightarrow \sO_{X_0}(2d-1,0) \rightarrow E \rightarrow \sO_{X_0}(2d-2,2d-1) \rightarrow 0.
\end{equation}
Twisting it by $\sO_{X_0}(-2d,-2d)$, we have
\[
h^1(X_0,E(-2d,-2d))=h^2(X_0,E(-2d,-2d))=0.
\]
Now twist \eqref{ext2} by $\sO_{X_0}(-d,-d)$ and consider the long exact sequence in cohomology. Since
\[
h^1(X_0,\sO_{X_0}(d-1,-d))=h^0(X_0,\sO_{X_0}(d-2,d-1))=d(d-1),
\]
we have $h^0(X_0,E(-d,-d))=h^1(X_0,E(-d,-d))$. So as soon as one of the cohomology groups vanishes, also the other does.
\end{proof}
We end this paper with an example of a code which allows us to construct Ulrich bundles on $(\xe,\sO_\xe(h))$ for an admissible Ulrich pair. We will use the resolution \eqref{risoluzionerulkernel}.
\begin{example}
In this example we will construct non-special rank two Ulrich bundles on $X_1$. In {\em Macaulay2}, given a divisor $D=t C_0+sf$ on the Hirzebruch surface $X_e$ the notation for line bundles is $\sO_\xe(D)=\sO_\xe(D C_0,Df)=\sO_\xe(s-et,t)$, i.e. $\sO_\xe(D)=\sO_\xe(a,b)$ where $a$ and $b$ are respectively the intersection between $D$ and the generators $C_0$ and $f$ of $\Pic(\xe)$.
\beginOutput
i1 : loadPackage "NormalToricVarieties";\\
\endOutput
Choose the self intersection invariant $e$
\beginOutput
i2 : e=1;\\
\endOutput
\beginOutput
i3 : FFe=hirzebruchSurface(e, CoefficientRing => ZZ/32003, Variable => y);\\
\endOutput
\beginOutput
i4 : S = ring FFe;\\
\endOutput
\beginOutput
i5 : loadPackage "BoijSoederberg";\\
\endOutput
\beginOutput
i6 : loadPackage "BGG";\\
\endOutput
\beginOutput
i7 : cohomologyTable(ZZ,CoherentSheaf,List,List):=(k,F,lo,hi)->(new Cohom $\cdot\cdot\cdot$\\
\endOutput

\vspace{0.2cm}

Fix a polarization $h=a C_0+bf$, the rank $r$ of our bundle and the first Chern class $u C_0+vf$.

\vspace{0.2cm}

\beginOutput
i8 : a=3;\\
\endOutput
\beginOutput
i9 : b=6;\\
\endOutput
\beginOutput
i10 : r=2;\\
\endOutput
\beginOutput
i11 : u=6;\\
\endOutput
\beginOutput
i12 : v=16;\\
\endOutput
\beginOutput
i13 : exp1=r*(2*b-1-e)-v-e*(r*(2*a-2)-u);\\
\endOutput
\beginOutput
i14 : exp2=r*(2*a-1)-u;\\
\endOutput
\beginOutput
i15 : exp3=r*(2*a+2*b-3-e)-u-v-e*(r*(2*a-2)-u);\\
\endOutput

\vspace{0.2cm}

we construct two random matrices to obtain, as in Theorem \ref{inversecondition}(2), two maps
\begin{gather*}
\sO_\xe^{exp1}(0,-1) \rightarrow \sO_\xe^{exp3}\\
\sO_\xe^{exp2} (-1,-e) \rightarrow \sO_\xe^{exp3}
\end{gather*}
\beginOutput
i16 : M=random(S^exp3,S^\{exp1:\{-1,0\}\});\\
\              4       2\\
o16 : Matrix S  <--- S\\
\endOutput
\beginOutput
i17 : P=random(S^exp3,S^\{exp2:\{0,-1\}\});\\
\              4       4\\
o17 : Matrix S  <--- S\\
\endOutput
\beginOutput
i18 : Mtot=M|P;\\
\              4       6\\
o18 : Matrix S  <--- S\\
\endOutput
\beginOutput
i19 : F=minimalPresentation ker Mtot;\\
\endOutput
\beginOutput
i20 : ShF= (sheaf(FFe,F))(2*b-1-e*(2*a-1),2*a-1);\\
\endOutput

\vspace{0.2cm}

Finally we check that the sheaf constructed in this way satisfies the vanishing of $H^0(\xe,F(-h))$ (or equivalently of $H^1(\xe,F(-h))$) required in Theorem \ref{inversecondition}(2).

\vspace{0.2cm}

\beginOutput
i21 : HH^0(FFe,ShF(e*a-b,-a))\\
o21 = 0\\
\        ZZ\\
o21 : ------module\\
\      32003\\
\endOutput
\beginOutput
i22 : exit\\
\endOutput
\end{example}

\bibliographystyle{spmpsci}

\end{document}